\documentclass[11pt,reqno]{amsart}
\usepackage{amsmath,latexsym,amssymb,verbatim,amsbsy,amsthm,mathrsfs,tikz,times}
\usepackage[margin=0.7in]{geometry}

\usepackage{color}
\usepackage[colorlinks=true, pdfstartview=FitV, linkcolor=blue, citecolor=blue,
urlcolor=blue]{hyperref}

\theoremstyle{plain}
\newtheorem{theorem}{Theorem}[section]
\newtheorem{definition}[theorem]{Definition}
\newtheorem{lemma}[theorem]{Lemma}
\newtheorem{proposition}[theorem]{Proposition}

\theoremstyle{definition}

\def\tilde{\widetilde}
\numberwithin{equation}{section}

\renewcommand\hat{\widehat}
\def\ZZ{{\mathbb Z}}
\def\RR{{\mathbb R}}
\def\TT{{\mathbb T}}

\renewcommand{\phi}{\varphi}

\title[Front Propagation for Nonlocal KPP Reaction-Diffusion Equations in Periodic Media]
{\bf Front Propagation for Nonlocal KPP Reaction-Diffusion Equations\\
in Periodic Media}
\author[P.~E. Souganidis]{Panagiotis E. $\text{Souganidis}^{a,b}$}
\thanks{$A$: Department of Mathematics,
The University of Chicago, 5734 S. University Avenue, Chicago, IL 60637,\\
\indent{}\indent{}E-mail: \texttt{souganidis@math.uchicago.edu}}
\thanks{$B$ Partially supported by the National Science Foundation
grants DMS-1266383 and DMS-1600129 and the Office for Naval Research\\
\indent{}\indent{}Grant
N00014-17-1-2095. }
\author[A. Tarfulea]{Andrei $\text{Tarfulea}^{c,d}$}
\thanks{$C$: Department of Mathematics,
The University of Chicago, 5734 S. University Avenue, Chicago, IL 60637,\\
\indent{}\indent{}E-mail: \texttt{atarfulea@math.uchicago.edu}}
\thanks{$D$ Partially supported by the National Science Foundation
Research Training Group grant DMS-1246999.}
\date{}

\begin{document}

\maketitle
\begin{abstract}
We study front propagation phenomena for a large class of nonlocal KPP-type reaction-diffusion
equations in oscillatory environments, which 
 model various forms of
population growth with periodic dependence. The nonlocal diffusion is an
anisotropic integro-differential
operator of order $\alpha \in (0,2)$.
\end{abstract}

\section{Introduction}

\noindent{}We study the long time/large space asymptotic behavior and front propagation for a class of
models governed by reaction-diffusion equations of the form
\begin{equation}
u_t + L^\alpha [u] = f(x,u) \ \text{ in } \ \RR^d \times [0,\infty)  \text{ and } \
u(\cdot,0) = u_0 \ \text{ on } \ \RR^d .
\label{main-eq}
\end{equation}
The reaction nonlinearity $f(x,u)$ is periodic with respect to  $x$ in the unit cube $Q$, and satisfies a
KPP-type condition in $u$; see \eqref{takis1}  and \eqref{takis2}.\\
\\
The diffusion $L^\alpha$ belongs to the nonlocal class of singular integral operators given by
\begin{equation}
L^\alpha[u](x) := \int (u(x) - u(x+y)) K(x,y) dy
\label{L-alpha}
\end{equation}
with ``mutation kernel'' $K$ that is positive, symmetric,  $1$-periodic in $x$ and has
a ``thick tail'' and is singular at $y=0$ of order $\alpha$; see  \eqref{takis3} and \eqref{takis4}.
Such operators generalize the fractional Laplacian $(-\Delta)^{\alpha/2}$.\\
\\
We also require that the linearized reaction-diffusion operator
$L^\alpha - \partial_u f$ has a negative principal eigenvalue $\lambda_1$; see assumption
\eqref{negative-eigenvalue}.\\
\\
The long-time/large-space behavior of the solution to \eqref{main-eq} is characterized by the
limiting properties and behavior as $\epsilon \rightarrow 0$ of
\begin{equation}
u^\epsilon(x,t) = u(\hat{x}|x|^{1/\epsilon},t/\epsilon) ,
\label{u-rescaled}
\end{equation}
which solves
\begin{equation}
 \epsilon u_t^\epsilon + \int \left( u(\hat{x}|x|^{1/\epsilon},t/\epsilon) -
u(\hat{x}|x|^{1/\epsilon} - y,t/\epsilon) \right) K(\hat{x}|x|^{1/\epsilon},y) dy
=f(\hat{x}|x|^{1/\epsilon}, u^\epsilon) \ \text{ in } \RR^d \times [0,\infty) ,
\nonumber
\end{equation}
with initial datum
$$ u^\epsilon(x,0) = u_0(\hat{x}|x|^{1/\epsilon}), $$
where, for $x \neq 0$,  $\hat{x} := x / |x|$ is the unit vector in the direction of $x$.\\
\\
We assume (see \eqref{initial-data}) that $u_0$ has sufficient decay at infinity 
so that $u^\epsilon(\cdot,0)$
converges to a ``patch'' function supported on the unit ball, that is,
\begin{align*}
\lim_{\epsilon \rightarrow 0} u^\epsilon(x,0) = \left\{
\begin{array}{ll}
u_0(0) \ & \text{ if } |x| < 1 , \\
u_0(x) \ & \text{ if } |x| = 1 , \\
0 \ & \text{ if } |x| > 1 .
\end{array}
\right.
\end{align*}
The rescaling \eqref{u-rescaled} is motivated by the scaling strategy of
M\'{e}l\'{e}ard and Mirrahimi \cite{MM}, which is used to study the behavior of the solution to
$$ u_t + (-\Delta)^{\alpha/2} u = u R \left( \int_{\RR} u(x,t) dx \right) \ \text{ in }
\ \RR \times [0,\infty) . $$
The arguments in  \cite{MM} exploit the one-dimensional nature
of the problem and the homogeneous in $x$  nature of the nonlocal diffusion and the nonlinearity.\\
\\
The results of this paper generalize to the much broader family
of equations \eqref{main-eq}, which have a richer range of behaviors and require much more
precise estimates.\\
\\
A key step in our analysis is the exponential transformation $$u^\epsilon= \exp (v^\epsilon),$$ which leads to
\begin{equation}
v_t^\epsilon + \int \left( 1-
\frac{\exp(v^\epsilon(\eta^\epsilon(x,y),t)/\epsilon)}{\exp(v^\epsilon(x,t)/\epsilon)} \right) K dy
= \frac{f(\hat{x}|x|^{1/\epsilon}, u^\epsilon)}{u^\epsilon};
\label{eq-v-eps}
\end{equation}
\\ 
above and henceforth for convenience of notation 
\begin{equation}
\eta^\epsilon(x,y) := \hat{(\hat{x}|x|^{1/\epsilon} - y)}
|\hat{x}|x|^{1/\epsilon}-y|^{\epsilon} \ \text{ and } \
K:= K(\hat{x} |x|^{1/\epsilon},y).
\label{eta-def}
\end{equation}
\\
\noindent Our first result concerns the behavior, as $\epsilon \rightarrow 0$, of $v^\epsilon$. Note that, in view of the spatially oscillatory behavior of \eqref{eq-v-eps}, the $v^\epsilon$ homogenizes in the limit. 
\begin{theorem}
Assume \eqref{takis1}, \eqref{takis2}, \eqref{takis3}, \eqref{takis4}, \eqref{negative-eigenvalue}, and \eqref{initial-data}.
Then, as $\epsilon \rightarrow 0$  and locally uniformly, the $v^\epsilon$'s  converge to the unique solution $v$   to  the variational inequality
\begin{equation}
\max (v_t - |\lambda_1|, v) = 0 \ \text{ in } \ \RR^d \times [0,\infty)  \text{ and } \
v(x,0) = \min (0,-(d+\alpha) \log (|x|)),
\label{homogenized-equation}
\end{equation}
given by 
\begin{equation}
v(x,t) = \min (0, |\lambda_1| t - (d+\alpha) \log(|x|)).
\label{homogenized}
\end{equation}
\label{main-theorem}
\end{theorem}
\noindent{}Knowing the asymptotic behavior of the $v^\epsilon$'s,  we obtain some concrete information about the
$u^\epsilon$'s. 
\\

\noindent For this, we recall that \eqref{main-eq} admits a positive periodic steady state
solution $u^+: \RR^d \rightarrow \RR$ solving
$$ L^\alpha[u^+] = f(x,u^+) \ \text{ in } \ \RR^d . $$
Although the existence of $u^+$ is classical, for completeness, we sketch a proof summary in the Appendix.\\

\noindent{}Our second result is:
\begin{theorem}
Assume \eqref{takis1}, \eqref{takis2}, \eqref{takis3}, \eqref{takis4}, \eqref{negative-eigenvalue}, and \eqref{initial-data}.
Then, as $\epsilon \rightarrow 0$ and locally uniformly,
\begin{align}
\left\{
\begin{array}{ll}
\medskip
u^\epsilon(x,t) \rightarrow 0 \ \text{ in } \ \lbrace{(x,t): \ |x|^{d+\alpha} > e^{|\lambda_1|t} \rbrace} ,\\[2.5mm]
\frac{u^\epsilon(x,t)}{u^+(\hat{x}|x|^{1/\epsilon})} \rightarrow 1 \ \text{ in } \ \lbrace{(x,t): \
|x|^{d+\alpha} < e^{|\lambda_1|t} \rbrace} .
\end{array}
\right.
\label{limit-for-u}
\end{align}
\label{main-corollary}
\end{theorem}
\noindent{}It follows from \eqref{limit-for-u} that, as $\epsilon \to 0$, 
\begin{equation*}
u^\epsilon \rightharpoonup \overline {U}:=  \frac{1}{|Q|} \int_Q u^+(x) dx.
\end{equation*}
We remark that, since $u^+$ is in general non-constant, we cannot expect that $u^\epsilon$ converges
in the inner region. \\
\\
Reaction-diffusion equations such as, for simplicity, the Fisher-KPP equation
\begin{equation}
u_t - \Delta u = f(u) \ \text{ in } \ \RR^d \times [0,\infty)
\label{Fisher-KPP}
\end{equation}
have been studied extensively, going back to the original work of Fisher \cite{Fish} and
Kolmogorov, Petrovskii, and Piskunov \cite{KPP}. We refer to Aronson and Weinberger \cite{AW} and \cite{AW2},
Bramson \cite{Bram}, Freidlin \cite{Freid}, Evans and Souganidis \cite{E-S}, and Majda and Souganidis \cite{MS}
for the derivations as well as up to date methods to study the long time/large-space asymptotic behavior.\\
\\
Equations such as \eqref{main-eq} arise from mesoscopic processes involving L\'{e}vy
flights seen in models of population dynamics and evolutionary ecology. For example,
Jourdain, M\'{e}l\'{e}ard, and Woyczynski \cite{JMW} show a rigorous derivation
when modeling Darwinian evolution of phenotypic variation, treating the population as
a stochastic point process; see also Fournier and M\'{e}l\'{e}ard \cite{F-M}.
When the probability distribution of mutations has a heavy tail
and belongs to the domain of attraction of a stable law, the corresponding diffusion admits jumps
(see Baeumer, Kovacs, and Meerschaert \cite{BKM} and Hillen and Othmer \cite{HO})
and, moreover, includes a spatial inhomogeneity. That is, the strength of the nonlocal diffusion
realistically depends on the genetic or spatial landscape, which motivates the treatment
of models with inhomogeneous reactions and diffusions.\\
\\
It turns out (see \cite{Freid}, \cite{E-S}, and \cite{MS})
that the asymptotic behavior of the solution to \eqref{Fisher-KPP} is recognized by a hyperbolic scaling
$(x,t) \rightarrow (x/\epsilon, t/\epsilon)$. This is consistent with the fact that the expected
advancing level sets (the fronts) move with algebraic (in time) velocity. That is, for every $h$ in the range
of $u$ and all $t$ sufficiently large, the set $\lbrace{ x \ : \ u(x,t) = h \rbrace}$ is comparable in a
quantified way to the set $\lbrace{|x| \approx c t^\gamma \rbrace}$ for some $c$; see \cite{Freid}, \cite{E-S}, and  \cite{MS} for the results for linear ($\gamma =1$) growth, and the recent works
of Berestycki, Mouhot, and Raoul \cite{BMR}, Bouin, Henderson, and Ryzhik \cite{BoHeRy},
and Henderson, Perthame, Souganidis \cite{HPS} for the superlinear ($\gamma > 1$) case.\\
\\
Many physical models, like, for example, neutron emission in nuclear physics, fast migrations in biology, and
convective heating in forest fires,  use processes  that run faster than what would be expected
for Brownian motion, either due to a fundamental sparseness of the medium or a long-range dispersion effect
in the measured physical quantity.
As a result the  nonlocal equations give rise to fronts moving with
exponential (in time) velocities. This behavior, that was widely known in the applied literature, was first shown rigorously by Cabr\'{e} and Roquejoffre
\cite{CR}, who considered the front-propagation and asymptotic dynamics of the nonlocal homogeneous problem
$$ u_t + A u = f(u) \ \text{ in } \ \RR^d \times [0,\infty), $$
with  $A$ the infinitesimal generator of a Feller semigroup and $f$ of KPP-type as in
\eqref{takis2}, but independent of $x$. \\
\\
In Cabr\'{e}, Coulon, and Roquejoffre \cite{CCR}, the authors study a simpler version of \eqref{main-eq}
with fractional Laplacian $L^\alpha = (-\Delta)^{\alpha/2}$ for $\alpha \in (0,1)$ and $f(x,u) =
\mu(x)u-u^2$, and  give an estimate on the exponential (in time) propagation for the low-value
level sets of $u$, that is, $\lbrace{u(x,t)=h \rbrace} \approx \lbrace{|x|^{d+\alpha} = c_\lambda
\exp (|\lambda_1| t) \rbrace}$ for all $h$ sufficiently small depending on $\mu$. For more general
KPP-type nonlinearities, Roquejoffre and Tarfulea \cite{R-me} further refine these estimates and show the
level sets symmetrize and the solutions become asymptotically flat over time.\\
\\
In this paper, we study the behavior of all level sets for general nonlocal
anisotropic $L^\alpha$ with no restriction on the order. In particular, the mutation kernel $K$ need not
be either translation invariant or homogeneous in $y$.\\
\\
While completing this paper, we became aware of a recent work of Bouin, Garnier, Henderson,
and Patout \cite{BGHP} that considers a locally uniform homogenization result for the
class of equations
$$ u_t = J*u-u + u(1-u) \ \text{ in } \ \RR \times [0,\infty) , $$
for kernels $J$ that are nonsingular, homogeneous in $x$, and monotone.
Although the assumptions of \cite{BGHP} allow for kernels with more general decay at infinity albeit one dimensional, our results
are focused on kernels with potential singularity at the origin and in any dimension. This generalizes the fractional Laplacian
as well as many other integro-differential operators, and introduces significant technical difficulties,
which we overcome here.\\
\\
The proof of Theorem \ref{main-theorem} is based on viscosity solution techniques originated in \cite{E-S},  \cite{E-S-2},
Barles, Evans, and Souganidis \cite{BES} and \cite{MS}.  The main
steps are $(i)$ obtaining a priori $L^\infty$-bounds on the $v^\epsilon$'s, $(ii)$ using
the half-relaxed upper and lower limits in conjunction with the perturbed test function methods to obtain
the equation satisfied by the limit (see \cite{E-S} and Evans and Souganidis \cite{E-S-2}),
and $(iii)$ deriving the claimed limit properties for the $u^\epsilon$'s.\\
\\
It turns out that steps $(i)$ and $(ii)$ above are connected with each other. Indeed, the obvious bounds implied
by \eqref{initial-data} (see Proposition \ref{prop-rough} and Lemma \ref{lem-ACC}) are not enough to yield
the limiting equation. We go around this difficulty by an iterative (inductive)
argument along which we improve at each step the upper and lower bounds (see Propositions
\ref{prop-induction-1} and \ref{prop-induction-2}), eventually obtaining the claimed result.

\subsection*{Organization of the paper}
\noindent{}The rest of the paper proceeds as follows.
In Section 2 we state all the relevant assumptions and 
prove a preliminary bound on the solution.
In Section 3, we state in Proposition \ref{prop-induction-1} and Proposition  \ref{prop-induction-2} the inductive arguments and give the proof of Theorem \ref{main-theorem}.
Proposition \ref{prop-induction-1} and Proposition  \ref{prop-induction-2} are shown, respectively, in
Section 4 and Section 5.
In Section 6 we give the proof of Theorem \ref{main-corollary}.
We also include an Appendix that contains some useful preliminary calculations and context. In 
Appendix A we recall the definitions of viscosity sub- and super-solutions, as well as
the half-relaxed upper and lower limits.
In Appendix C we show the existence of a positive periodic first eigenfunction $e^g$ to the linearized
operator $L^\alpha - \partial_u f$, as well as a positive steady state $u^+$ when the
corresponding eigenvalue is negative. Finally,
in Appendix C we prove a technical lemma used in Section 2.

\subsection*{Notation} We write $Q$ and $B_r(x)$ respectively  for the unit cube and the open ball of radius $r$ and center $x$  in $\RR^d$. Given $A, B \in \RR$, $A\lesssim B$ means that there exists some constant $c>0$, which is independent of the various parameters, such that $A\leq B$.  Throughout the paper, $\|l\|$ denotes the sup-norm of a given bounded function $l$. We denote by $E^C$ the complement of $E\subset \RR^d.$

\subsection*{Terminology}  A function which is periodic on the unite cube $Q$ is referred to as $1$-periodic. Throughout the paper sub-and super-solutions are understood in the viscosity sense.

\section{Assumptions and Preliminary Results}
\noindent{}For the reaction nonlinearity $f$, we assume that, for each $u\in \RR$,  
\begin{equation}\label{takis1}
x \to f(x,u) \ \text{ is $1$-periodic}, 
\end{equation}
and that there exists a constant $M > 0$, which is independent of $x$, such that, for all $x\in\RR^d$, 
\begin{equation}\label{takis2}
s \rightarrow f(x,s)/s \text{ is decreasing in } \ s,  \
f(x,0) = 0, \ \text{ and } \ f(x,\cdot ) \leq 0 \ \text{in} \ [M,\infty).
\end{equation}
\noindent Observe that, in view of \eqref{takis1} and \eqref{takis2},
we may write
\begin{equation}
f(x,u) = \mu(x)u - E(x,u) ,
\label{f-KPP}
\end{equation}
where $\mu(x): = \partial_u f(x,0)$ and $E(x,u)$ is an error term such that, for some $\overline {M}>0$ 
and $\bar{m}>0$,
\\
\begin{equation}
\overline{m} u^2 \leq E(x,u) =: \partial_u f(x,0)u - f(x,u) \leq \overline{M} u^2 .
\label{E-assumption}
\end{equation}
\\
As far as the mutation kernel $K$ is concerned,
we assume that 
\\
\begin{equation}\label{takis3}
K \ \text{ is positive, $1$-periodic and $C^2$ \ in  \ $ x$, 
and  \ symmetric in \ $y$,}
\end{equation}
\\
and there exists $C_K>0$ such that, for $\mathcal{K} = K, \ |D_x K|, \  \ |D^2_x K|$ and all $x, y \in \RR^d,$
\begin{equation}\label{takis4}
C^{-1}_K \leq \mathcal{K}(x,y)|y|^{d+\alpha} \leq C_K .
\end{equation}
\\
The algebraic decay of the tail of $K$, seen in the upper and lower bounds of  \eqref{takis4}
for $|y|$ large, will be very important. Indeed, this is the mechanism which produces exponential
(in time) propagation of fronts. \\
\\
We remark, however, that
the proofs of our main results do not rely on $K$ having a singularity as $|y| \rightarrow 0$. As a matter of fact
the symmetric singularity at $y=0$ allowed by \eqref{takis4} is a technical obstacle which we overcome
with careful estimates. Actually  all the arguments in the proofs of Theorem \ref{main-theorem} and  Theorem \ref{main-corollary}
remain valid if, instead  of  \eqref{takis4}, we assume 
$$\frac{C_K^{-1}}{1+ |y|^{d+\alpha}} \leq \mathcal{K} \leq \frac{C_K}{|y|^{d+\alpha}} . $$
The singular lower bound of \eqref{takis4} is only needed to prove the existence of
a positive principal eigenfunction to the linearized operator, where we need to compare
$\langle L^\alpha[u],u \rangle$ with a Sobolev norm; see \eqref{negative-eigenvalue} below and
Lemma \ref{prop-eigen-2}.\\
\\
The symmetry assumption on $K$ allows us to write, after  a change of variables,
\begin{equation}
L^\alpha[u](x) = \frac{1}{2} \int (2u(x) - u(x+y) - u(x-y))K(x,y)dy .
\label{symmetrized-L}
\end{equation}
Moreover, we assume that the bottom of the spectrum for the linearized stationary
operator is negative, that is,
\begin{equation}
\text{there exists a positive, $1$-periodic } \ e^g \ \text{ satisfying } \
(L^\alpha - \mu) e^g = \lambda_1 e^g \ \text{ with } \ \lambda_1 < 0 .
\label{negative-eigenvalue}
\end{equation}
For $L^\alpha$ defined as in \eqref{L-alpha}, a positive (and unique) first eigenfunction $e^g$ always exists.
For completeness, we include a sketch of the proof of this in the Appendix (Proposition \ref{prop-eigen}).\\
\\
As for the role of the principal eigenvalue $\lambda_1$ of the linearized operator, we recall that,
for \eqref{main-eq} with $L^\alpha$ the fractional Laplacian,
Berestycki, Roquejoffre, and Rossi \cite{BRR} show that
the bottom of the spectrum of the operator $(-\Delta)^{\alpha/2}-\partial_u f$
is the indicator to a drastic shift in the asymptotic behavior for \eqref{main-eq}.
If $\lambda_1 \geq 0$, the model exhibits extinction, that is, $u(\cdot,t) \rightarrow 0$
as $t \rightarrow \infty$. On the other hand, if $\lambda_1 < 0$, the model exhibits invasion, that is,
$u$ converges, as $t\to \infty$ uniformly on compact
sets,  to a unique positive steady solution.
The sign of $\lambda_1$ therefore measures the strengths of the depletion and growing zones 
$\lbrace{ x \ : \ \partial_u f(x,0) \leq 0 \rbrace}$ and 
 $\lbrace{ x \ : \ \partial_u f(x,0) > 0 \rbrace}$ 
with respect to the ambient diffusion. Since we wish to study the
nontrivial asymptotic behavior of the solutions to \eqref{main-eq}, we
must work with models that exhibit invasion; hence the need for assumption \eqref{negative-eigenvalue}.
Theorem \ref{main-corollary} then extends the invasion results of \cite{BRR} to a much larger class
of models.\\
\\
Lastly,  we assume that
$u_0\in C(\RR^d)$  and there exist positive constants $c_1$ and $c_2$ such that, 
for all $x \in \RR^d$,
\begin{equation}
\frac{c_1}{1+|x|^{d+\alpha}} \leq u_0(x) \leq \frac{c_2}{1+|x|^{d+\alpha}} .
\label{initial-data}
\end{equation}
\\
An immediate consequence of \eqref{takis1}, \eqref{takis2}, \eqref{takis3},  \eqref{takis4}, \eqref{negative-eigenvalue}, and \eqref{initial-data} are
the bounds given in the next proposition.
\begin{proposition}
Let $u$ be the solution to \eqref{main-eq} 
and assume  \eqref{takis1}, \eqref{takis2}, \eqref{takis3}, \eqref{takis4}, \eqref{negative-eigenvalue}, 
and \eqref{initial-data}.
Then there exist positive constants $B_0 > |\lambda_1|$, $c_0 < C_0$, and $A_0$ such that,
for all $(x,t)$,
\begin{equation}
\frac{c_0 e^{-A_0 t}}{1+e^{-|\lambda_1| t}|x|^{d+\alpha}} \leq u(x,t)
\leq \frac{C_0}{1+e^{-B_0 t}|x|^{d+\alpha}} .
\label{rough-bounds}
\end{equation}
\label{prop-rough}
\end{proposition}
\noindent{}The estimate follows from the fact that, for some positive constants $c, c'$ and $c''$ and all $x$ and $u$,
\begin{equation}\label{takis10}
-cu - cu^2 \leq f(x,u) \leq c'u - c'' u^2.
\end{equation}
Then Proposition \ref{prop-rough} is shown by establishing that the lower and upper bounds of
\eqref{rough-bounds} are respectively sub- and super-solutions to \eqref{main-eq}.\\
\\
To prove Proposition \ref{prop-rough} and throughout the paper we will require the following technical lemma which is proved in Appendix C.
\begin{lemma} Assume \eqref{takis1}, \eqref{takis2}, \eqref{takis3},  and \eqref{takis4}, and let
$ h(x,t) = \frac{1}{1+e^{-\lambda t} |x|^{d+\alpha}} . $
Then there exists a constant $D>0$ depending only on $d$ and $K$ such that, for all $(x,t)$,
\begin{equation}
\left| L^\alpha [h] \right| \leq
\frac{D e^{-\alpha \lambda t /(d+\alpha)}}{1+e^{-\lambda t}|x|^{d+\alpha}} .
\label{ACC-bound}
\end{equation}
\label{lem-ACC}
\end{lemma}
\begin{proof}[Proof of Proposition \ref{prop-rough}:]
\noindent{}It follows from \eqref{initial-data} that \eqref{rough-bounds}
holds at $t=0$ for an appropriate choice of $c_0$ and $C_0$.\\
\\
Let  $W(x,t) = C_0 / (1+e^{-B_0 t} |x|^{d+\alpha})$ and
$w(x,t) = c_0 e^{-A_0 t} / (1+e^{-|\lambda_1| t} |x|^{d+\alpha})$. Since
\eqref{main-eq} satisfies a comparison principle, we show that $W$ and $w$ are in fact
super- and sub-solutions to \eqref{main-eq} for an appropriate choice of constants.\\
\\
Set $\mu_+ := \sup_{x} \mu(x)$ and $\mu_- := \inf_{x} \mu(x)$. We remark that $\mu_+ >
|\lambda_1|$ and that $\mu_-$ might be negative.
It follows from \eqref{f-KPP} and \eqref{E-assumption} that
$$ \mu_- u - \overline{M}u^2 \leq f(x,u) \leq \mu_+ u - \overline{m} u^2 . $$
For  $B_0 > D + \mu_+ > |\lambda_1|$ and $C_0> \max(1,B_0/\overline{m})$ large enough
that $u_0< W(\cdot,0)$, Lemma \ref{lem-ACC} yields
$$ W_t + L^\alpha[W] \geq
\frac{B_0 C_0 e^{-B_0 t}|x|^{d+\alpha}}{(1+e^{-B_0 t}|x|^{d+\alpha})^2}
- \frac{D C_0 e^{-\alpha B_0 t/(d+\alpha)}}{1+e^{-B_0 t}|x|^{d+\alpha}}
= \left( B_0 - D e^{-\alpha B_0 t/(d+\alpha)} \right) W
- \frac{B_0}{C_0} W^2 \geq f(x,W) , $$
which implies the upper bound of \eqref{rough-bounds}.\\
\\
The lower bound follows, after choosing $A_0 > |\lambda_1| + D - \mu_-$ and
$c_0 < |\lambda_1|/\overline{M}$ sufficiently small so that $w(\cdot,0) < u_0$, from the estimate
\begin{align*}
w_t &+ L^\alpha[w] \leq
\frac{|\lambda_1| c_0 e^{-(|\lambda_1|+A_0)t}|x|^{d+\alpha}}{(1+e^{-|\lambda_1| t}|x|^{d+\alpha})^2}
-\frac{A_0 c_0 e^{-A_0 t}}{1+e^{-|\lambda_1| t}|x|^{d+\alpha}}
+\frac{D c_0 e^{-A_0 t} e^{-\alpha |\lambda_1| t/(d+\alpha)}}{1+e^{-|\lambda_1| t}|x|^{d+\alpha}} \\
&= \left( |\lambda_1|-A_0+D e^{-\alpha |\lambda_1| t / (d+\alpha)} \right) w
- \frac{|\lambda_1|}{c_0} e^{A_0 t} w^2 \leq f(x,w) .
\end{align*}
\end{proof}

\section{Improvement Iteration and Proof of Theorem \ref{main-theorem}}
\noindent{}Although \eqref{rough-bounds}  implies the bound 
\begin{equation}
\epsilon \log (c_0) - A_0 t - \epsilon \log (1+e^{-|\lambda_1| t/\epsilon} |x|^{(d+\alpha)/\epsilon}) \leq
v^\epsilon(x,t) \leq
\epsilon \log (C_0) - \epsilon \log (1+e^{-B_0 t/\epsilon} |x|^{(d+\alpha)/\epsilon}), 
\label{rough-v-bounds-hyp}
\end{equation}
this is obviously insufficient to obtain \eqref{homogenized}.
It does, however, yield information about the half-relaxed upper and lower limits $v^*$ and $v_*$.\\
\\
If $|x|^{d+\alpha} \leq e^{B_0 t}$, \eqref{rough-v-bounds-hyp} yields
$v^*(x,t) \leq 0$. However, if $|x|^{d+\alpha} > e^{B_0 t}$, letting
$\epsilon \rightarrow 0$ in \eqref{rough-v-bounds-hyp} shows that $v^*(x,t) \leq B_0 t - (d+\alpha) \log (|x|)$.
A similar analysis extends the lower bounds of \eqref{rough-v-bounds-hyp} to $v_*$. Putting it all together,
we see that \eqref{rough-v-bounds-hyp} yields
\begin{equation}
\min \left(0, |\lambda_1| t - (d+\alpha) \log (|x|) \right)-A_0 t \leq
v_*(x,t) \leq v^*(x,t) \leq \min \left( 0, B_0 t - (d+\alpha) \log (|x|) \right),
\label{induction-base}
\end{equation}
and, letting $t \rightarrow 0$, we obtain
\begin{equation}
v_*(x,0) = v^*(x,0) = \min(0, -(d+\alpha) \log (|x|)) .
\label{initial-data-limit}
\end{equation}
We remark that, in general, the half-relaxed upper and lower  limits of solutions satisfy initial and boundary value
conditions in a relaxed sense; see Barles and Perthame \cite{B-P}, Ishii \cite{Ish},
and \cite{BES}. However, the inequalities in  \eqref{rough-v-bounds-hyp}
hold uniformly, so that we can extract \eqref{initial-data-limit}.\\
\\
We note for future use that, for any fixed $M > 0$, there exists a $C_M>0$ such that, for all $(y,s) \in \RR^d \times [0,\infty),$
\begin{equation}
\frac{C_M^{-1}}{1+e^{-M s}|y|^{d+\alpha}}
\leq \exp \left( \min (0,M s - (d+\alpha) \log |y|) \right)
\leq \frac{C_M}{1+e^{-M s}|y|^{d+\alpha}} .
\label{basic-bound}
\end{equation}
Having established \eqref{induction-base} as a base case, we now describe the following inductive procedure.\\
\\
Let  $\lbrace{A_k \rbrace}_{k \geq 0}$ and
$\lbrace{B_k \rbrace}_{k \geq 0}$ be two monotone decreasing sequences in $(0,\infty)$ with $A_0$ and $B_0$ given by the exponents
seen in \eqref{rough-bounds},  
$$A_k \rightarrow 0 \ \text{and}  \ B_k \rightarrow |\lambda_1|,$$
and, for each $k \geq 0$,
\begin{equation}\label{b-ratio}
\frac{B_k}{B_{k+1}} < 1 + \frac{\alpha}{d}, \ B_{k+1} \geq |\lambda_1| , \ A_k - A_{k+1}
< \frac{|\lambda_1|}{4(d+\alpha)} \ \text{ and} \ A_{k+1} \geq 0.
\end{equation}
We will show that, if $A_k$ and $B_k$ are such that 
\begin{equation}
\min \left(0, |\lambda_1| t - (d+\alpha) \log (|x|) \right)-A_k t \leq
v_*(x,t) \leq v^*(x,t) \leq \min \left( 0, B_k t - (d+\alpha) \log (|x|) \right) ,
\label{rough-v-bounds}
\end{equation}
then the same holds for $A_{k+1}$ and $B_{k+1}$.\\
\\
This is the conclusion of the following two propositions 
that establish appropriate variational inequalities for the half-relaxed upper and lower limits of $v^\epsilon$.
\begin{proposition}
Assume \eqref{b-ratio}, \eqref{rough-v-bounds}, and let $V_k(x,t) := \min(0,B_k t - (d+\alpha) \log (|x|))$.
Then
\begin{equation}
\min ( v^*_t - B_{k+1} ,v^* - V_{k+1} ) \leq 0 \ \ \text{in} \ \  \RR^d \times (0,\infty).
\label{improved-bounds-upper}
\end{equation}
\label{prop-induction-1}
\end{proposition}

\begin{proposition}
Assume \eqref{b-ratio}, \eqref{rough-v-bounds}, and let $W_k(x,t) := \min(0,|\lambda_1| t -
(d+\alpha) \log (|x|)) - A_k t$. Then
\begin{equation}
\max( v_{*,t} - |\lambda_1| + A_{k+1}, v_* - W_{k+1}) \geq 0 \ \ \text{in} \ \  \RR^d \times (0,\infty).
\label{improved-bounds-lower}
\end{equation}
\label{prop-induction-2}
\end{proposition}
\noindent{}The proof of Theorem \ref{main-theorem} is then completed with an argument by
induction and the following technical lemma, which is an immediate consequence of the
definition of viscosity solution.
\begin{lemma}
Let $U \subset \RR^d$ be open  and $F, \ G:U \times [0,\infty) \rightarrow \RR$ be respectively upper- and lower-semicontinuous.
Assume that, for a fixed $C>0$ and all $(x,t) \in U \times [0,\infty)$,  $F(x,t) \leq Ct$ and $G(x,t) \geq -Ct$, and,
$$\min(F_t,F) \leq 0 \ \ \text{and} \ \  \max(G_t,G) \geq 0 \ \ \text{in the viscosity sense.}$$
Then
$$F \leq 0 \ \text{ and } \ G \geq 0 \ \ \text{on} \ \ U \times [0,\infty). $$
\label{lem-induction-complete}
\end{lemma}
\begin{proof}[The proof of Theorem~\ref{main-theorem}]
It follows from Lemma \ref{lem-induction-complete} that, for
all $(x,t) \in \RR^d \times [0,\infty)$,
$$ \min \left(0, |\lambda_1| t - (d+\alpha) \log (|x|) \right)-A_{k+1} t \leq
v_*(x,t) \leq v^*(x,t) \leq \min \left( 0, B_{k+1} t - (d+\alpha) \log (|x|) \right).$$
Since $A_k \rightarrow 0$ and $B_k \rightarrow |\lambda_1|$,
we conclude that for all $(x,t)\in \RR^d\times [0,\infty)$
$$ v_*(x,t) = v^*(x,t) = \min(0,|\lambda_1|t-(d+\alpha)\log(|x|)).$$
This in fact shows that $v^\epsilon$ has a locally uniform limit $v$, which is given by \eqref{homogenized},
and clearly satisfies \eqref{homogenized-equation}.
\end{proof}

\section{Proof of Proposition \ref{prop-induction-1}}
\noindent{}Since the proof is long and technical, we begin with a general outline.
Let $\psi$ be a smooth test function that touches $v^*$ from above
at a point $(x_0,t_0) \in \RR^d\times (0,\infty)$ and,  arguing by contradiction, we assume that
\begin{equation}
v^*(x_0,t_0) > V_{k+1}(x_0,t_0) = \min(0, B_{k+1} t_0 - (d+\alpha) \log (|x_0|)) ,
\label{b-bound}
\end{equation}
and
\begin{equation}
\psi_t (x_0,t_0) > B_{k+1} .
\label{subsolution-v-upper-not}
\end{equation}
Then we modify  $\psi$ into a new function $\psi^\epsilon$ that converges
uniformly as $\epsilon \to 0$ to $\psi$ in a neighborhood of $(x_0,t_0)$, and we show that,
for all $\epsilon$ sufficiently small,
each $\psi^\epsilon$ is a super-solution to \eqref{eq-v-eps}
in a neighborhood of $(x_0,t_0)$ which is independent of  $\epsilon$.
Using $v^\epsilon$ as a test function will then show that $v^\epsilon$ must remain a fixed
distance away from $\psi^\epsilon$ in the neighborhood of $(x_0,t_0)$. Therefore
$\psi(x_0,t_0)$ is strictly larger than $v^*(x_0,t_0)$, which contradicts the assumption.
\\ \\
Although the above gives an accurate summary of the scheme of the proof, for the actual argument
we need to make a few technical modifications, which, however, do not change the overall strategy of the proof.
Firstly, it is necessary to change $\psi$ outside of a neighborhood of $(x_0,t_0)$.
Secondly, we will not be able to test $\psi^\epsilon$
against $v^\epsilon$ directly, since we cannot assume that $v^\epsilon \leq v^*$.
Instead, we must modify $v^\epsilon$ outside of a large but fixed compact set $K$, and then
appeal to the locally uniform subconvergence of $v^\epsilon$ to $v^*$; see \eqref{loc-unif-limits}.
Since \eqref{eq-v-eps} is nonlocal, this will create some
error terms  that must be bounded.\\
\\ 
In proving that $\psi^\epsilon$ is a super-solution,  we will  need a bound on the ratio
between $\psi^\epsilon$ at an ``arbitrary'' point and $\psi^\epsilon$ near $(x_0,t_0)$. This is used
to estimate the size of the nonlocal diffusion term, and employs 
\eqref{eigen-intermediate-1} and \eqref{eigen-intermediate-1-alternate}. Then
\eqref{b-bound} acts as an impromptu (local) lower bound
for $\psi$ at $(x_0,t_0)$. The lower bound will then hold for $\psi^\epsilon$ on a sufficiently small
neighborhood, depending on $\psi$ but not on $\epsilon$. Combined
with the global upper bounds of \eqref{rough-v-bounds}, the effect of the nonlocal diffusion is seen
to decay to zero as $\epsilon \rightarrow 0$,
provided \eqref{b-ratio} holds for $B_{k+1}$ and $B_k$.\\
\\
Moreover, the estimates for the nonlocal diffusion will require two slightly different approaches
depending on the value of $\alpha$. To that end, we first prove a useful decomposition
for the operator $a \rightarrow L^\alpha[a]/a$.\\
\\
In view of \eqref{f-KPP},  if $e^g$ is the first eigenfunction for the linearized operator $L^\alpha - \mu$,
then $ L^\alpha[e^g]/e^{g} = -|\lambda_1| + \mu$.\\
\\
 Given  $a_0: \RR^d \times [0,\infty) \rightarrow \RR$,
let $ a(x,t) = a_0(x,t) + \epsilon g (\hat{x}|x|^{1/\epsilon})$. The decomposition depends on whether $\alpha  \in (0,1)$ or $\alpha \in [1,2)$.\\
\\
In the former case, recalling   
\eqref{eta-def}, we  find
\begin{align*}
\frac{L^\alpha[e^a]}{e^a} &= \int \left(1-\frac{\exp(a(\eta^\epsilon(x,y),t)/\epsilon)}
{\exp(a/\epsilon)} \right) K dy \nonumber \\
&= \frac{L^\alpha[e^g]}{e^g} + \int \left( \frac{\exp(g(\hat{x}|x|^{1/\epsilon}-y))}{\exp(g(\hat{x}|x|^{1/\epsilon}))}
- \frac{\exp(a(\eta^\epsilon(x,y),t)/\epsilon)}{\exp(a/\epsilon)} \right) Kdy, \\
\end{align*}
and, hence,
\begin{equation}\label{eigen-intermediate-1-alternate}
\frac{L^\alpha[e^a]}{e^a}= = \frac{L^\alpha[e^g]}{e^{g}}
+ \int G_\epsilon(x,y) \left(
1 - \frac{\exp(a_0(\eta^\epsilon(x,y),t)/\epsilon)}{\exp(a_0/\epsilon)}
\right) K dy,
\end{equation}
\\
where
\begin{equation}
G_\epsilon(x,y) := \exp \left( g(\hat{x}|x|^{1/\epsilon} - y) -
g(\hat{x}|x|^{1/\epsilon}) \right)
\label{eigen-intermediate-3}
\end{equation}
is bounded above and below by positive constants
uniformly in $\epsilon$.\\
\\
When $\alpha \in [1,2)$, we
need a stronger cancellation at $y=0$ to handle the potential singularity of $K$. 
It follows from \eqref{symmetrized-L} that, for any $R>0$, 
\begin{align*} 
\frac{L^\alpha[e^a]}{e^a} &=
\int_{|y|<R} \left(1 -
\frac{\exp (a(\eta^\epsilon(x,y),t)/\epsilon) +
\exp (a(\eta^\epsilon(x,-y),t)/\epsilon)}{2 \exp (a/\epsilon)}
\right) K dy \nonumber \\
& \hspace{2cm} +\int_{|y|\geq R} \left( 1-
\frac{\exp(a(\eta^\epsilon(x,y),t)/\epsilon)}{\exp(a/\epsilon)}
\right) K dy \nonumber \\
\end{align*}
and, hence,
\begin{equation} \label{eigen-intermediate-1}
\frac{L^\alpha[e^a]}{e^a}=\begin{cases}
 \frac{L^\alpha[e^g]}{e^{g}}
+ \frac{1}{2} \int_{|y|<R} \left( H_\epsilon(a_0,x,y)
 + H_\epsilon(a_0,x,-y) \right) K dy  \\[2.5mm]
+\int_{|y| \geq R} G_\epsilon(x,y) \left(
1 - \frac{\exp(a_0(\eta^\epsilon(x,y),t)/\epsilon)}{\exp(a_0/\epsilon)}
\right) K dy,
\end{cases}
\end{equation}
where
\begin{equation}
H_\epsilon(a_0,x,y) :=
G_\epsilon(x,y)
\left(1 - \frac{\exp(a_0(\eta^\epsilon(x,y),t)/\epsilon)}{\exp(a_0/\epsilon)} \right) .
\label{eigen-intermediate-2}
\end{equation}
\\
\begin{proof}[The proof of  Proposition \ref{prop-induction-1}]
We assume that \eqref{b-ratio} and \eqref{rough-v-bounds}
hold for $v^*$ and the relevant constants and observe that \eqref{rough-v-bounds-hyp} 
implies $v^* \leq 0$ independently of
Proposition \ref{prop-induction-1}, that is, \eqref{improved-bounds-upper} trivially holds
when $|x| \leq 1$. As such, we only consider points $(x,t)$ with
$|x|>1$ and $t>0$.\\
\\
Let $\psi$ be a smooth function such that
$v^* - \psi$ assumes a global maximum value of $0$ at $(x_0,t_0)$ and
\begin{equation}
\psi(x_0,t_0) = v^*(x_0,t_0) > \min(0,B_{k+1} t_0 -(d+\alpha) \log(|x_0|))
\label{first-condition-on-psi}
\end{equation}
and, for some $\sigma > 0$,
$$ \psi_t(x_0,t_0) -B_{k+1} = \sigma > 0 . $$
Let $r_0 = \min(t_0, |x_0|-1)/3$ and consider the cylinder
$$ D_{r_0}(x_0,t_0) := B_{r_0}(x_0) \times (t_0 - r_0, t_0 + r_0)
\subset \lbrace{(x,t): \  |x|>1 , \ t>0 \rbrace} . $$
We modify $\psi$ outside a fixed neighborhood of $(x_0,t_0)$. Without loss of
generality, in light of \eqref{rough-v-bounds}, we may assume that
\begin{equation}
\psi(x,t) \leq \min(0,B_k t - (d+\alpha) \log(|x|)) .
\label{psi-is-good-upper-bound}
\end{equation}
Let $\delta>0$ be such that
\begin{equation}
\frac{B_k d}{d+\alpha} + \delta < B_{k+1} , \ e^{4 \delta r_0} < (|x_0|-r_0)^{\alpha/2},
\ \text{ and } \ \delta < \frac{\sigma}{3} .
\label{delta}
\end{equation}
Notice that, if $\delta$ and $r_0$ satisfy \eqref{delta}, then so do any $\delta' < \delta$ and $r_0' < r_0$.
We can therefore reduce the values of $\delta$ and $r_0$ without violating the above assumptions.\\
\\
Taking $\delta$ and $r_0$ smaller, depending on $\psi$ and $\sigma$ but not on $\epsilon$, we
construct a smooth  function $\theta$ such that, for some $C>0$ depending on $r_0$, $\psi$, $B_{k+1}$, and $B_k$ but not on $\epsilon$,
\begin{align}
\left\{
\begin{array}{ll}
\medskip
(i) \ & \theta = \psi \ \text{ on } \ D_{r_0/2}(x_0,t_0) , \\
\medskip
(ii) \ & \theta = \psi+\max(0,\delta(t-t_0+2r_0))
\ \text{ on } \ D_{r_0}(x_0,t_0)^C , \\
\medskip
(iii) \ & \theta \leq \psi+\max(0,\delta(t-t_0+2r_0))
\ \text{ everywhere } , \\
\medskip
(iv) \ & \theta \geq \min(0,B_{k+1} t - (d+\alpha) \log (|x|)) - \delta r_0 \ \text{ on }
\ D_{r_0}(x_0,t_0) , \\
\medskip
(v) \ & \theta_t - B_{k+1} > 2 \sigma / 3 \ \text{ on } \
D_{r_0}(x_0,t_0),\\
(vi) \ & \| D \theta \|  \  \text{and}  \ \| D^2 \theta \|< C. 
\end{array}
\right.
\label{theta-ident-1}
\end{align}
In light of \eqref{b-bound} and the regularity of $\psi$,
such a $\theta$ always exists; note that to satisfy \eqref{theta-ident-1}$(iv)$ and
$(v)$ it may be necessary to reduce the size of $r_0$ from its initial value,
but this only depends on $\psi$ and not on $\epsilon$.\\
\\
Next we show that 
$$ \psi^\epsilon(x,t) := \theta(x,t) +\epsilon g(\hat{x}|x|^{1/\epsilon}) +
\delta (t_0-r_0-t), $$
is a super-solution to \eqref{main-eq} in $D_{r_0}(x_0,t_0)$.\\
\\
Let $\phi$ be a smooth test function such that $\psi^\epsilon
- \phi$ assumes a minimum value of $0$ at $(\bar{x},\bar{t}) \in D_{r_0}(x_0,t_0)$;
recall  that $|\bar{x}| > 1$. Since $\phi = \psi^\epsilon$ at $(\bar{x},\bar{t})$, for all $y \in \RR^d$, we have 
$$ \frac{\exp(\phi(\eta^\epsilon(\bar{x},y), \bar{t})/\epsilon)}{\exp(\phi(\bar{x},\bar{t})/\epsilon)} =
\frac{\exp(\psi^\epsilon(\eta^\epsilon(\bar{x},y), \bar{t})/\epsilon)}{\exp(\psi^\epsilon(\bar{x},\bar{t})/\epsilon)}
+ \frac{\exp(\phi(\eta^\epsilon(\bar{x},y),\bar{t})/\epsilon) - 
\exp(\psi_\epsilon(\eta^\epsilon(\bar{x},y),\bar{t})/\epsilon)}{\exp(\phi(\bar{x},\bar{t})/\epsilon)} , $$
and, hence,
\begin{equation}
1 - \frac{\exp(\phi(\eta^\epsilon(\bar{x},y), \bar{t})/\epsilon)}{\exp(\phi(\bar{x},\bar{t})/\epsilon)}
\geq 1 - \frac{\exp(\psi^\epsilon(\eta^\epsilon(\bar{x},y), \bar{t})/\epsilon)}
{\exp(\psi^\epsilon(\bar{x},\bar{t})/\epsilon)} .
\label{supersol-pre-1}
\end{equation}
Moreover, it follows from \eqref{theta-ident-1}$(v)$ that
\begin{equation}
\phi_t (\bar{x}, \bar{t}) \geq \psi_t^\epsilon(\bar{x}, \bar{t}) =
\theta_t(\bar{x}, \bar{t}) - \delta >
\psi_t(\bar{x}, \bar{t}) - \sigma/3 > B_{k+1}+\sigma/3 .
\label{supersol-pre-2}
\end{equation}
Using \eqref{supersol-pre-1} and \eqref{supersol-pre-2}, for any $R > 0$, we find that, at
$(\bar{x},\bar{t})$,
\begin{equation}\label{supersol-main}
\begin{cases}
& \phi_t +\int_{B_R(\bar{x})} \left(1- \frac{\exp(\phi(\eta^\epsilon
(\bar{x},y),\bar{t})/\epsilon)}{\exp(\phi/\epsilon)} \right) K dy
+\int_{B_R(\bar{x})^c} \left(1- \frac{\exp(\psi^\epsilon(\eta^\epsilon(\bar{x},y),\bar{t})
/\epsilon)}{\exp(\psi^\epsilon/\epsilon)} \right) K dy \\
& \hspace{4mm} \geq \psi_t-\frac{\sigma}{3}+ \int \left(1- \frac{\exp(\psi^\epsilon
(\eta^\epsilon(\bar{x},y),\bar{t})/\epsilon)}{\exp(\psi^\epsilon/\epsilon)}
\right) K dy > B_{k+1} +\frac{\sigma}{3}+ \frac{L^\alpha[e^{\psi^\epsilon}]}
{e^{\psi^\epsilon}} .
\end{cases}
\end{equation}
Next we employ  \eqref{eigen-intermediate-1-alternate} and \eqref{eigen-intermediate-1}
with $\psi^\epsilon$ in place of $a$ to get, for any $r = r(\epsilon) > 0$,
\begin{equation}
\frac{L^\alpha[e^{\psi^\epsilon}]}{e^{\psi^\epsilon}} =
-|\lambda_1|+\mu(\hat{\bar{x}}|\bar{x}|^{1/\epsilon}) +
\mathcal{I}_{\epsilon,\alpha}
+\int_{|y| \geq r} G_\epsilon(\bar{x},y) \left(1- \frac{\exp(\theta(\eta^\epsilon(\bar{x},y),
\bar{t})/\epsilon)}{\exp(\theta/\epsilon)} \right) K dy \
\label{supersol-rhs}
\end{equation}
where
\begin{align*}
\mathcal{I}_{\epsilon,\alpha} := \left\{
\begin{array}{ll}
\medskip
\int_{|y| < r} G_\epsilon(\bar{x},y) \left(1- \frac{\exp(\theta(\eta^\epsilon(\bar{x},y),
\bar{t})/\epsilon)}{\exp(\theta(\bar{x},\bar{t})/\epsilon)} \right) K dy \  & \text{ if } \ \alpha < 1 , \\
\medskip
\frac{1}{2} \int_{|y| < r} (H_\epsilon(\theta,\bar{x},y)+H_\epsilon(\theta,\bar{x},-y)) K dy
& \text{ if } \ \alpha \geq 1 .
\end{array}
\right.
\end{align*}
Take $r(\epsilon) = \epsilon |\bar{x}|^{1/\epsilon} $ and recall \eqref{eigen-intermediate-2}.
It follows from \eqref{theta-ident-1}$(iv)$ that
$$\theta(\bar{x},\bar{t}) \geq \min (0, B_{k+1} \bar{t} - (d+\alpha) \log (|\bar{x}|))-\delta r_0 . $$
For the rest of the proof, let 
$$ h(y) = \frac{\exp(\theta(\eta^\epsilon(\bar{x},y),\bar{t})/\epsilon)}
{\exp(\theta(\bar{x},\bar{t})/\epsilon)} . $$
Since \eqref{theta-ident-1}$(iii)$ and \eqref{psi-is-good-upper-bound} hold for all $(x,t)$ and $\bar{t} > t_0-r_0$,
we note that \eqref{basic-bound} implies that, for all $y \in \RR^d$,
\begin{equation}\label{h-bound}
\begin{cases}
h(y) &\leq \displaystyle \frac{\exp (\min (0,B_k \bar{t}/\epsilon -(d+\alpha) \log (|\hat{\bar{x}}
|\bar{x}|^{1/\epsilon} - y|) + \delta(\bar{t}-t_0+2r_0)/\epsilon)}
{\exp (\min (0, B_{k+1} \bar{t}/\epsilon-(d+\alpha) \log (|\bar{x}|)/\epsilon) -
\delta r_0 / \epsilon )}\\[3mm]
 & \displaystyle\lesssim
\frac{1 + e^{-B_{k+1} \bar{t}/\epsilon} |\bar{x}|^{(d+\alpha)/\epsilon}}
{1+ e^{-B_k \bar{t}/\epsilon} |\hat{\bar{x}} |\bar{x}|^{1/\epsilon}-y|^{d+\alpha}}
e^{\delta (\bar{t}-t_0+3 r_0)/\epsilon}.
\end{cases}
\end{equation}
and  proceed to estimate the last term of \eqref{supersol-rhs}.\\
\\
Recalling \eqref{delta}, and the facts that  $G_\epsilon$ is bounded uniformly in $\epsilon$ and $\bar{t} < t_0+r_0$ and $|\bar{x}| > |x_0|-r_0$,
we observe that, if $e^{B_{k+1} \bar{t}} \geq |\bar{x}|^{d+\alpha}$, then $h(y) < 2 e^{4 \delta r_0/ \epsilon}$ and therefore
$$ \int_{|y| \geq r} G_\epsilon(\bar{x},y)(1-h(y)) Kdy \gtrsim -r^{-\alpha}
e^{4 \delta r_0 / \epsilon} \gtrsim -\epsilon^{-\alpha}|\bar{x}|^{-\alpha/(2 \epsilon)} . $$
Otherwise, if $e^{B_{k+1}\bar{t}} < |\bar{x}|^{d+\alpha}$, we use the change of variables
$y = |\bar{x}|^{1/\epsilon} z$ to get
\begin{align}
\int_{|y| \geq r} & G_\epsilon(\bar{x},y)(1-h(y)) Kdy \gtrsim -e^{-B_{k+1} \bar{t}/\epsilon}
|\bar{x}|^{(d+\alpha)/\epsilon} \int_{|y| \geq r} \frac{e^{\delta(\bar{t}-t_0+3r_0)/\epsilon}}{1+e^{-B_k \bar{t}/\epsilon}
|\hat{\bar{x}}|\bar{x}|^{1/\epsilon} -y|^{d+\alpha}} \frac{dy}{|y|^{d+\alpha}} \nonumber \\
&= -e^{-B_{k+1} \bar{t}/\epsilon} |\bar{x}|^{d/\epsilon}
\int_{|z| \geq \epsilon} \frac{e^{\delta(\bar{t}-t_0+3r_0)/\epsilon}}{1+ e^{-B_k \bar{t}/\epsilon}
|\bar{x}|^{(d+\alpha)/\epsilon}|\hat{\bar{x}}-z|^{d+\alpha}} \frac{dz}{|z|^{d+\alpha}}
\nonumber \\
& \geq -\epsilon^{-d-\alpha} e^{-B_{k+1} \bar{t}/\epsilon} |\bar{x}|^{d/\epsilon}
\int \frac{e^{\delta(\bar{t}-t_0+3r_0)/\epsilon} dw}{1+e^{-B_k \bar{t}/\epsilon}
|\bar{x}|^{(d+\alpha)/\epsilon}|w|^{d+\alpha}} \nonumber \\
&\gtrsim
-\epsilon^{-d-\alpha} \exp \left( \left(\frac{B_k d}{d+\alpha}-B_{k+1} \right) \frac{\bar{t}}{\epsilon}+
\delta \frac{\bar{t}-t_0+3r_0}{\epsilon} \right) .
\nonumber
\end{align}
For the above inequalities, the first step used \eqref{takis4}, the second employed H\"{o}lder's inequality,
and in the third we made the implicit substitution 
$$w = (e^{B_k \bar{t} / \epsilon}
|\bar{x}|^{(d+\alpha)/\epsilon} )^{-1/(d+\alpha)} \tilde{w}.$$
Notice that \eqref{delta} and the fact that $3r_0 \leq t_0$ imply the exponent is negative. Therefore,
in both cases, the last term of \eqref{supersol-rhs} becomes nonnegative as
$\epsilon \rightarrow 0$.\\
\\
It remains to show that \begin{equation} \label{takis40} \lim_{\epsilon \to 0} |\mathcal{I}_{\epsilon,\alpha}|=0, \end{equation} and for this we need to consider two cases depending on whether $\alpha \in [1,2)$ or $\alpha \in (0,1)$.\\
\\
For the former we must estimate the function $h$ and its derivatives near $y = 0$.
To that end, observe the facts that
$$ |\partial_i G_\epsilon(\bar{x},y)| = |G_\epsilon(\bar{x},y) \partial_i
g (\hat{\bar{x}} |\bar{x}|^{1/\epsilon} - y)|
\leq \| G_\epsilon(\bar{x},\cdot) \| \| D g \|, $$
and, if $|y| \leq r \leq |\bar{x}|^{1/\epsilon} / 2$, then 
$$ \frac{1}{2} |\bar{x}|^{1/\epsilon} \leq |\hat{\bar{x}}|\bar{x}|^{1/\epsilon}-y|
\leq \frac{3}{2} |\bar{x}|^{1/\epsilon}.$$
Next we need to compute a number of derivatives that enter in the argument below. Recall that, 
for $1 \leq j \leq d$, 
$$ \eta^\epsilon_j(x,y) = (\hat{x}_j |x|^{1/\epsilon} - y_j) |\hat{x}
|x|^{1/\epsilon} - y |^{\epsilon-1}. $$
Then
$$ \partial_i \eta^\epsilon_j(\bar{x},y) = -\delta_{ij} |\hat{\bar{x}}|\bar{x}|^{1/\epsilon}
-y|^{\epsilon-1} + \eta^\epsilon_j(\bar{x},y) (1-\epsilon)
(\hat{\bar{x}}_i|\bar{x}|^{1/\epsilon}-y_i)|\hat{\bar{x}}|\bar{x}|^{1/\epsilon}-y|^{-2} , $$
with
$$ |\partial_i \eta^\epsilon_j (\bar{x},y)| \leq 4|\bar{x}|^{1-1/\epsilon} , $$
and
\begin{align}
\partial_{ik}^2 \eta^\epsilon_j(\bar{x},y) &= \delta_{ij}(1-\epsilon) \eta^\epsilon_k(\bar{x},y)
|\hat{\bar{x}}|\bar{x}|^{1/\epsilon}-y|^{-2} +(1-\epsilon) \partial_k \eta^\epsilon_j(\bar{x},y)
(\hat{\bar{x}}_i|\bar{x}|^{1/\epsilon}-y_i)|\hat{\bar{x}}|\bar{x}|^{1/\epsilon}-y|^{-2}
\nonumber \\
&-\delta_{ik}(1-\epsilon) \eta^\epsilon_j(\bar{x},y) |\hat{\bar{x}}|\bar{x}|^{1/\epsilon}-y|^{-2}
-2 \eta^\epsilon_j(\bar{x},y) (1-\epsilon)\frac{(\hat{\bar{x}}_i|\bar{x}|^{1/\epsilon}-y_i)
(\hat{\bar{x}}_k|\bar{x}|^{1/\epsilon}-y_k)}{|\hat{\bar{x}}|\bar{x}|^{1/\epsilon}-y|^{4}}
, \nonumber
\end{align}
with
$$ |\partial_{ik}^2 \eta^\epsilon_j(\bar{x},y)| \leq 24 |\bar{x}|^{1 - 2/\epsilon} . $$
It follows that
\begin{align}
\partial_i h(y) &= \frac{1}{\epsilon} h(y) D \theta (\eta^\epsilon(\bar{x},y),\bar{t})
\cdot \partial_i \eta^\epsilon (\bar{x},y) \nonumber \\
&= h(y) \left( -\frac{1}{\epsilon} \frac{\partial_i \theta(\eta^\epsilon(\bar{x},y),\bar{t})}
{|\hat{\bar{x}}|\bar{x}|^{1/\epsilon} - y|^{1-\epsilon}} + \left(\frac{1}{\epsilon}-1 \right)
D \theta(\eta^\epsilon(\bar{x},y),\bar{t}) \cdot \eta^\epsilon(\bar{x},y)
\frac{\hat{\bar{x}}_i|\bar{x}|^{1/\epsilon}-y_i}
{|\hat{\bar{x}}|\bar{x}|^{1/\epsilon}-y|^2} \right), \nonumber
\end{align}
and, for $|y| \leq r$,
\begin{align}
| \partial_{ik}^2 h(y)| &\lesssim \frac{ \sup_{|z|<r} h(z) }{\epsilon} \left( \frac{1}{\epsilon}
\| D \theta \|^2 \| D \eta^\epsilon \|^2 +
\| D^2 \theta \|_{L^\infty} \| D \eta^\epsilon \|^2 +
\| D \theta \| \| D^2 \eta^\epsilon \|\right)
\nonumber \\
&\lesssim \left( \sup_{|z|<r} h(z) \right) |\bar{x}|^{-2/\epsilon} \left( \frac{|\bar{x}|^2}{\epsilon^2}
+ \frac{|\bar{x}|}{\epsilon} \right) .
\nonumber
\end{align}
Combining all the above and using \eqref{theta-ident-1}(vi), we get
$$ G_\epsilon(\bar{x}, \pm y) = 1 + O(|y|) , $$
and
$$ h(\pm y) = 1 \pm D h(0) \cdot y+\left( \sup_{|z|<r} h(z) \right) O \left(|\bar{x}|^{-2/\epsilon}
\left( \frac{|\bar{x}|^2}{\epsilon^2} + \frac{|\bar{x}|}{\epsilon} \right) |y|^2 \right) . $$
Note that
$$ D h(0) = -\frac{1}{\epsilon} D \theta (\bar{x},\bar{t}) |\bar{x}|^{1-1/\epsilon}
+\left(\frac{1}{\epsilon}-1\right)(D \theta(\bar{x},\bar{t}) \cdot x) \hat{\bar{x}} |\bar{x}|^{-1/\epsilon} . $$
It follows from \eqref{h-bound} that, for all $y \in \RR^d$,
\begin{align}
|h(y)| &\leq \frac{1 + e^{-B_{k+1} \bar{t}/\epsilon} |\bar{x}|^{(d+\alpha)/\epsilon}}
{1+ e^{-B_k \bar{t}/\epsilon} |\hat{\bar{x}} |\bar{x}|^{1/\epsilon}-y|^{d+\alpha}}
e^{\delta(\bar{t}-t_0+3r_0)/\epsilon} \nonumber \\
&\leq (1+ \min(e^{-B_{k+1} \bar{t}/\epsilon}|\bar{x}|^{(d+\alpha)/\epsilon},
2^{d+\alpha} e^{(B_k-B_{k+1}) \bar{t}/\epsilon})) e^{\delta(\bar{t}-t_0+3r_0)/\epsilon} ,
\label{h-bound-1-alter}
\end{align}
where the last inequality used that, for any $(\bar{x},\bar{t})$ with $|\bar{x}|>1$ and for all
$\epsilon$ sufficiently small,
$$ \frac{1 + e^{-B_{k+1} \bar{t}/\epsilon} |\bar{x}|^{(d+\alpha)/\epsilon}}
{1+ e^{-B_k \bar{t}/\epsilon} |\hat{\bar{x}} |\bar{x}|^{1/\epsilon}-y|^{d+\alpha}}
\leq 1 + e^{-B_{k+1} \bar{t}/\epsilon} |\bar{x}|^{(d+\alpha)/\epsilon} , $$
and
$$ \frac{1 + e^{-B_{k+1} \bar{t}/\epsilon} |\bar{x}|^{(d+\alpha)/\epsilon}}
{1+ e^{-B_k \bar{t}/\epsilon} |\hat{\bar{x}} |\bar{x}|^{1/\epsilon}-y|^{d+\alpha}}
\leq \frac{|\bar{x}|^{-(d+\alpha)/\epsilon}+e^{-B_{k+1} \bar{t}/\epsilon}}
{|\bar{x}|^{-(d+\alpha)/\epsilon}+2^{-(d+\alpha)}e^{-B_k \bar{t}/\epsilon}}
\leq 1 + 2^{d+\alpha} e^{(B_k-B_{k+1}) \bar{t} / \epsilon} . $$
Since the estimates above hold simultaneously, we can bound the ratio by their minimum.
Using \eqref{eigen-intermediate-1} we ontain
\begin{align}
\mathcal{I}_{\epsilon,\alpha} &= \frac{1}{2}
\int_{|y| < r} (H_\epsilon(\theta,\bar{x},y)+H_\epsilon(\theta,\bar{x},-y)) K dy +
\int_{|y| < r} \left( G_\epsilon(\bar{x},y)(1-h(y))
+ G_\epsilon(\bar{x},-y)(1-h(-y)) \right) K dy \nonumber \\
&= \int_{|y| < r} D h(0) \cdot y (G_\epsilon(\bar{x},-y) - G_\epsilon(\bar{x},y)) K dy +
\int_{|y| < r} \left(\sup_{|z|<r} h(z) \right) O \left( |\bar{x}|^{-2/\epsilon}
\left( \frac{|\bar{x}|^2}{\epsilon^2} + \frac{|\bar{x}|}{\epsilon} \right)
|y|^2 \right) K dy \nonumber \\
&= \int_{|y|<r} O\left( | D h(0)| |y|^2 +\left(\sup_{|z|<r} h(z) \right) |\bar{x}|^{-2/\epsilon}
\left( \frac{|\bar{x}|^2}{\epsilon^2} + \frac{|\bar{x}|}{\epsilon} \right)
|y|^2 \right) K dy \nonumber
\end{align}
so that
\begin{equation}
| \mathcal{I}_{\epsilon,\alpha}| \lesssim \left( |\bar{x}|^{1-1/\epsilon} +\left(\sup_{|z|<r} h(z) \right)
|\bar{x}|^{-2/\epsilon}
\left( \frac{|\bar{x}|^2}{\epsilon^2} + \frac{|\bar{x}|}{\epsilon} \right) \right)
r^{2-\alpha} .
\label{I-1}
\end{equation}
Observe that, since $|\bar{x}|>1$, $|\bar{x}|^{-1/\epsilon} r^{2-\alpha} \rightarrow 0$
for all $\alpha \geq 1$.\\
\\
It then follows from \eqref{I-1} and \eqref{h-bound-1-alter} that, for all $\epsilon$ sufficiently small and $\alpha \in [1,2)$,
$$| \mathcal{I}_{\epsilon,\alpha}| \lesssim O(\epsilon) + e^{\delta(\bar{t}-t_0+3r_0)/\epsilon}
\min (e^{-B_{k+1} \bar{t}}|\bar{x}|^d, e^{(B_k-B_{k+1})\bar{t}}|\bar{x}|^{-\alpha})^{1/\epsilon},$$
while using \eqref{b-ratio} and \eqref{delta} we find 
uniformly for $(\bar{x},\bar{t}) \in D_{r_0}(x_0,t_0)$ and $\epsilon > 0$,
$$ \min (e^{-B_{k+1} \bar{t}}|\bar{x}|^d, e^{(B_k-B_{k+1})\bar{t}}|\bar{x}|^{-\alpha}) <
e^{-\delta \bar{t}} \leq e^{-\delta (\bar{t}-t_0+3r_0)},$$
that is, since
$$ \frac{B_{k+1}}{d} > \frac{B_k - B_{k+1}}{\alpha} + \delta \frac{d+\alpha}{d \alpha} , $$
we have either
$$ |\bar{x}| < e^{B_{k+1} \bar{t}/d}e^{-\delta \bar{t}/d} $$
or
$$|\bar{x}| > e^{(B_k-B_{k+1})\bar{t}/\alpha}e^{\delta \bar{t}/\alpha},$$
and, hence, $\underset{\epsilon \to 0}\lim \ |\mathcal{I}_{\epsilon,\alpha}| =0$.\\
\\
The proof  of \eqref{takis40} when $\alpha \in (0,1)$ follows almost immediately from the bounds computed
above. Indeed, in light of  \eqref{eigen-intermediate-1-alternate}, we have
\begin{equation}
|\mathcal{I}_{\epsilon,\alpha}| = \left| \int_{|y|<r} G_\epsilon(\bar{x},y) (1-h(y)) K dy \right| \lesssim
\left(\sup_{|z|<r} h(z) \right) \int_{|y|<r} O(|y|) K dy
\lesssim \left(\sup_{|z|<r} h(z) \right) \frac{|\bar{x}|^{1-1/\epsilon}}{\epsilon} r^{1-\alpha}
\label{I-2}
\end{equation}
and, hence,
$$ |\mathcal{I}_{\epsilon,\alpha}| \lesssim e^{\delta(\bar{t}-t_0+3r_0)/\epsilon}
\min (e^{-B_{k+1} \bar{t}}|\bar{x}|^d, e^{(B_k-B_{k+1})\bar{t}}|\bar{x}|^{-\alpha})^{1/\epsilon} . $$
Arguing as before and using \eqref{h-bound-1-alter}, we now conclude.\\
\\
We note that the decay of $\mathcal{I}_{\epsilon,\alpha}$ and some of the estimates on $h$ are dictated by the decay of
$|\bar{x}|^{-1/\epsilon}$ and $e^{-\bar{t}/\epsilon}$, which is not uniform for arbitrary $(\bar{x},\bar{t})$.
However, $(x_0,t_0)$ was fixed with $|x_0|>1$ and
$t_0>0$, and $r_0 \leq \min (|x_0|-1,t_0)/3$. Therefore, the decay of $\mathcal{I}_{\epsilon,\alpha}$ to zero
(and any bound on $h$) is in fact uniform for any $(\bar{x},\bar{t}) \in D_{r_0}(x_0,t_0)$.\\
\\
It follows that, if $B_{k+1} \geq |\lambda_1|$, then for all $\epsilon$ sufficiently small
(and independent of $\phi$) \eqref{supersol-main} becomes
$$ \phi_t (\bar{x},\bar{t})+\int_{B_R(\bar{x})} \left(1- \frac{\exp(\phi(\eta^\epsilon
(\bar{x},y),\bar{t})/\epsilon)}{\exp(\phi(\bar{x},\bar{t})/\epsilon)} \right) K dy
+\int_{B_R(\bar{x})^c} \left(1-h(y) \right) K dy
\geq \mu(\hat{\bar{x}}|\bar{x}|^{1/\epsilon}) + \frac{\sigma}{3} , $$
which yields that $\psi^\epsilon$ is a super-solution. \\
\\
Next we  derive a contradiction by comparing $\psi^\epsilon$ and $v^\epsilon$.
There is, however, an additional problem.  Since we do not know that $v^\epsilon \leq \psi^\epsilon \ \text{in} \ \RR^d \times (0,\infty), $ we cannot use
$v^\epsilon$ in place of $\phi$ in definition \eqref{supersolution}.
The inductive hypothesis
\eqref{rough-v-bounds} only holds for $v^*$ and $v_*$. These are locally uniform
one-sided limits, so the best we can obtain for $v^\epsilon$ is \eqref{loc-unif-limits} (one-sided
limits that are uniform on compact sets).\\
\\
Consider the compact set $K = \overline{B}_{R_K}(0) \times [0,t_0+r_0+1]$ with
$$R_K := 3 \exp \left( \frac{(B_0+4\delta)(t_0+r_0)}{d+\alpha} \right)
(|x_0|+r_0) + 1 . $$
Then, recalling \eqref{psi-is-good-upper-bound}, for all $\epsilon$ sufficiently small and all $(x,t) \in K$,
$$v^\epsilon < v^* + \delta r_0 \leq \psi + \delta r_0
\leq \min(0,B_k t - (d+\alpha) \log (|x|)) + \delta r_0 . $$
Given $v^\epsilon$, let $\bar{v}^\epsilon$ be a smooth function such that
\begin{equation}\label{v-eps-bar-1}
\bar{v}^\epsilon = v^\epsilon \ \text{ in } \ \overline{B}_{R_K-1}(0) \times [0,t_0+r_0], \  \bar{v}^\epsilon \leq \psi \ \text{ in } \ K^C , \ \text{and} \  \bar{v}^\epsilon \leq v^\epsilon \ \text{ in } \ \RR^d \times [0,\infty) .
\end{equation}
Then, for all $(x,t) \in D_{r_0}(x_0,t_0)^C$, \eqref{theta-ident-1}$(ii)$ implies
$$ \psi^\epsilon = \theta + \epsilon g(\hat{x}|x|^{1/\epsilon}) + \delta(t_0-r_0-t) >
\psi + \delta r_0 > \bar{v}^\epsilon . $$
It follows that, if there ever comes a first
time $\bar{t} \leq t_0+r_0$ where $\underset{x\in \RR^d} \min(\psi^\epsilon(\cdot, \bar{t})-\bar{v}^\epsilon(\cdot, \bar{t})) = 0$,
it must happen at a point $\bar{x} \in B_{r_0}(x_0)$. Then $\bar{v}^\epsilon$ is an admissible test function
(on $B_{r_0}(x_0) \times [t_0-r_0, \bar{t}]$) and must satisfy
\begin{equation}
\bar{v}_t^\epsilon (\bar{x},\bar{t})+\int_{B_R(\bar{x})} \left(1-
\frac{\exp(\bar{v}^\epsilon(\eta^\epsilon (\bar{x},y),\bar{t})/\epsilon)}
{\exp(\bar{v}^\epsilon(\bar{x},\bar{t})/\epsilon)} \right) K dy
+\int_{B_R(\bar{x})^c} \left(1-h(y) \right) K dy
\geq \mu(\hat{\bar{x}}|\bar{x}|^{1/\epsilon}) + \frac{\sigma}{3} .
\label{v-bar-upper}
\end{equation}
However, in view of  \eqref{eq-v-eps},
$$ v_t^\epsilon (\bar{x},\bar{t})+\int \left(1-
\frac{\exp(v^\epsilon(\eta^\epsilon (\bar{x},y),\bar{t})/\epsilon)}
{\exp(v^\epsilon(\bar{x},\bar{t})/\epsilon)} \right) K dy =
\mu(\hat{\bar{x}}|\bar{x}|^{1/\epsilon}) - E e^{-v^\epsilon(\bar{x},\bar{t})/\epsilon} , $$
or
\begin{equation}
\bar{v}_t^\epsilon (\bar{x},\bar{t})+\int \left(1-
\frac{\exp(\bar{v}^\epsilon(\eta^\epsilon (\bar{x},y),\bar{t})/\epsilon)}
{\exp(\bar{v}^\epsilon(\bar{x},\bar{t})/\epsilon)} \right) K dy +
\mathcal{J}_\epsilon = \mu(\hat{\bar{x}}|\bar{x}|^{1/\epsilon}) - E e^{-\bar{v}^\epsilon
(\bar{x},\bar{t})/\epsilon} ,
\label{v-bar-upper-2}
\end{equation}
where
$$ \mathcal{J}_\epsilon :=
\int \frac{\exp(\bar{v}^\epsilon(\eta^\epsilon (\bar{x},y),\bar{t})/\epsilon)
-\exp(v^\epsilon(\eta^\epsilon (\bar{x},y),\bar{t})/\epsilon)}
{\exp(\bar{v}^\epsilon(\bar{x},\bar{t})/\epsilon)} K dy . $$
Observe that, for all $\epsilon$ sufficiently small depending on $x_0$ and $t_0$,
$$ \lbrace{ y: \ |\eta^\epsilon(\bar{x},y)| > R_K-1 \rbrace} =
\lbrace{ y: \ |\hat{\bar{x}}|\bar{x}|^{1/\epsilon}-y| > (R_K-1)^{1/\epsilon} \rbrace}
\subseteq S_\epsilon , $$
where
$$ S_\epsilon := \lbrace{ y: \ |y| > (R_K-1)^{1/\epsilon} - |\bar{x}|^{1/\epsilon} \rbrace} \subset
\lbrace{ y: \ |y| > 2 |\bar{x}|^{1/\epsilon} e^{\frac{(B_0+4\delta)\bar{t}}
{(d+\alpha) \epsilon}} \rbrace} . $$
Since $v^\epsilon$ and $\bar{v}^\epsilon$ are equal on $S_\epsilon^C$, we would like
to subtract the two equations above. This creates some error terms in the nonlocal piece, but $R_K$
is sufficiently large that the errors vanish as $\epsilon \rightarrow 0$. Subtracting
\eqref{v-bar-upper-2} from \eqref{v-bar-upper} yields
\begin{equation}
\int_{B_R(\bar{x})^c} \frac{\exp(\bar{v}^\epsilon(\eta^\epsilon (\bar{x},y),\bar{t})/\epsilon)-
\exp(\psi^\epsilon(\eta^\epsilon (\bar{x},y),\bar{t})/\epsilon)}
{\exp(\bar{v}^\epsilon(\bar{x},\bar{t})/\epsilon)} Kdy - \mathcal{J}_\epsilon \geq
\frac{\sigma}{3} + E e^{-v^\epsilon(\bar{x},\bar{t})/\epsilon} > 0 .
\label{to-contradict}
\end{equation}
The first integral term, however, is nonpositive. Moreover, by \eqref{rough-v-bounds-hyp}
$$ -\mathcal{J}_\epsilon \leq \int_{S_\epsilon} \frac{\exp(v^\epsilon(
\eta^\epsilon(\bar{x},y),\bar{t})/\epsilon)}{\exp(\psi^\epsilon(\bar{x},\bar{t})/\epsilon)} K dy
\lesssim \int_{S_\epsilon} \frac{\left(1+e^{-B_0 \bar{t}/\epsilon}|\hat{\bar{x}}
|\bar{x}|^{1/\epsilon}-y|^{d+\alpha}
\right)^{-1}}{\exp(\min(0,B_{k+1} \bar{t}
-(d+\alpha) \log (|\bar{x}|))/\epsilon)} e^{\frac{4 \delta r_0}{\epsilon}} K dy . $$
Then \eqref{basic-bound} and the choice of $R_K$ yields
$$ -\mathcal{J}_\epsilon \lesssim e^{\frac{4 \delta r_0}{\epsilon}} \int_{S_\epsilon}
\frac{1+e^{-B_{k+1} \bar{t}/\epsilon}|\bar{x}|^{(d+\alpha)/\epsilon}}
{e^{-B_0 \bar{t}/\epsilon}|\hat{\bar{x}}|x|^{1/\epsilon}-y|^{d+\alpha}}
K dy \lesssim \int_{S_\epsilon} \frac{dy}{|y|^{d+\alpha}} \lesssim
2^{-\alpha} |\bar{x}|^{-\alpha/\epsilon} e^{-\frac{\alpha (B_0+4\delta) \bar{t}}
{(d+\alpha) \epsilon}} . $$
The middle inequality used the fact that $|\hat{\bar{x}}|\hat{x}|^{1/\epsilon}-y|^{d+\alpha}
\geq |\bar{x}|^{(d+\alpha)/\epsilon} e^{(B_0+4 \delta)(t_0+r_0)/\epsilon}$ on $S_\epsilon$,
which becomes arbitrarily large as $\epsilon \rightarrow 0$.\\
\\
It follows that the left side of \eqref{to-contradict}
becomes nonpositive as $\epsilon \rightarrow 0$, yet the right side is at least $\sigma/3$. This is a contradiction,
so we conclude that, for all $t < t_0 + r_0$, all $|x| > 1$, and all $\epsilon$ sufficiently small,
$$\bar{v}^\epsilon < \psi^\epsilon = \theta + \epsilon g(\hat{x}|x|^{1/\epsilon})
-\delta (t - t_0+r_0) . $$
This would imply that $v^*(x_0,t_0) \leq \psi(x_0,t_0) - \delta r_0$, since  $\theta = \psi$
and $\bar{v}^\epsilon = v^\epsilon$ inside of $D_{r_0/2}(x_0,t_0)$),
a contradiction to \eqref{first-condition-on-psi}. This shows that \eqref{subsolution-v-upper-not} cannot hold,
thus proving \eqref{improved-bounds-upper} and Proposition \ref{prop-induction-1}.
\end{proof}

\section{Proof of Proposition \ref{prop-induction-2}}
\noindent{}We follow  a strategy similar to the  proof  of Proposition \ref{prop-induction-1}.
There are, however, some key differences, which we explain next.\\
\\
Recall that, when $|x| \leq 1$, $v^*(x,t) \leq 0$ for all $t>0$ followed immediately from \eqref{induction-base}.
The rest of the proof of Proposition \ref{prop-induction-1} 
assumed that $|x|>1$, which was crucial for bounding the nonlocal
diffusion; see \eqref{I-1} and \eqref{I-2}.\\
\\
 A similar bound will be used in the proof of
Proposition \ref{prop-induction-2} and will also assume that $|x|>1$.
However, \eqref{induction-base} by itself does not yield an adequate lower bound
on $v_*$ for $|x| \leq 1$, so this case must be treated in a different way.\\
\\
We first prove \eqref{improved-bounds-lower} on $B_1^C \times (0,\infty) $ without assuming an  improved lower
bound on $v_*$  in $\overline B_1 \times (0,\infty)$.
From this, we extract a lower bound for $u$, the solution to \eqref{main-eq} in
unscaled coordinates, which then lets us prove \eqref{improved-bounds-lower} in $\overline B_1 \times (0,\infty)$
via a nonlocal maximum principle; see Lemma \ref{lem-inner-lower-bound}.

\begin{proof}[The proof of Proposition \ref{prop-induction-2}]
Assume that $\psi$ is a smooth
test function that touches $v_*$ from below at a point $(x_0,t_0)$ with $|x_0| > 1$ and $t_0 > 0$.
Arguing by contradiction, we assume that
\begin{equation}
\psi(x_0,t_0) = v_*(x_0,t_0) < -A_{k+1} t_0 + \min(0, |\lambda_1|t_0 - (d+\alpha) \log (|x_0|)) = W_{k+1}(x_0,t_0) ,
\label{a-bound}
\end{equation}
and, for some $\sigma > 0$,
\begin{equation}
\psi_t (x_0,t_0) - |\lambda_1| + A_{k+1} = -\sigma .
\label{supersolution-v-lower-not}
\end{equation}
We need to modify this test function outside of a neighborhood of $(x_0,t_0)$, but the
analysis overall is simpler than for Proposition \ref{prop-induction-1}.\\
\\
Let $r_0 = \min(t_0,|x_0|-1)/3$. and define $D_{r_0}(x_0,t_0)$ as before.
Without loss of generality, in light of
\eqref{rough-v-bounds}, we may assume that
\begin{equation}
\min(0,|\lambda_1|t-(d+\alpha)\log(|x|)) - A_{k+1} t \geq
\psi(x,t) \geq \min(0,|\lambda_1|t-(d+\alpha)\log(|x|)) - A_k t .
\label{psi-is-good-lower-bound}
\end{equation}
Let $\delta > 0$ be such that
\begin{equation}
A_k - A_{k+1} + \delta < \frac{|\lambda_1|}{4(d+\alpha)} \ \text{ and } \ \delta <
\min \left( \frac{\sigma}{3} , \frac{A_{k+1} (t_0-r_0)}{3r_0} \right);
\label{delta-2}
\end{equation}
notice that, if $\delta$ and $r_0$ satisfy \eqref{delta-2}, then so do any $\delta'<\delta$ and $r_0'<r_0$.
We can therefore reduce the values of $\delta$ and $r_0$ without violating
\eqref{delta-2}.\\
\\
Taking $r_0$ smaller, depending on $\psi$ but not
$\epsilon$, we construct a smooth function $\theta$ satisfying, for some $C>0$ independent of $\epsilon$,
\begin{align}
\left\{
\begin{array}{ll}
\medskip
(i) \ & \theta = \psi \ \text{ on } \ D_{r_0/2}(x_0,t_0) , \\
\medskip
(ii) \ & \theta = \psi-\max(\delta r_0,\delta(t-t_0+2r_0)) \ \text{ on }
\ D_{r_0}(x_0,t_0)^C, \\
\medskip
(iii) \ & \theta \geq \psi -\max(\delta r_0,\delta(t-t_0+2r_0)) \
\text{ everywhere }, \\
\medskip
(iv) \ & \theta \leq \min(0,|\lambda_1|t-(d+\alpha)\log (|x|))-A_{k+1} t+\delta r_0 < 0 \ \text{ on }
\ D_{r_0}(x_0,t_0), \\
\medskip
(v) \ & \theta_t-|\lambda_1|+A_{k+1} < -2 \sigma / 3 \ \text{ on } \ D_{r_0}(x_0,t_0), \\
\medskip
(vi) \ & \| D \theta \|, \ \| D^2 \theta \|< C . 
\end{array}
\right.
\label{theta-ident-2}
\end{align}
The existence of such a $\theta$ follows from \eqref{rough-v-bounds} and the fact that $\psi$ is strictly below $v_*$  away from  $(x_0,t_0)$;
note that it may be necessary to choose  $r_0$ and $\delta$ smaller.
We also remark that \eqref{theta-ident-2}$(ii)$ and \eqref{psi-is-good-lower-bound} yield that the upper bound
\eqref{theta-ident-2}$(iv)$ holds
for all $(x,t)\in \RR^d \times [0, t_0+r_0]$.\\
\\
Recall that $e^g$ is the principal eigenfunction for the linearized operator and define
$$ \psi^\epsilon(x,t):= \theta(x,t)+\epsilon g(\hat{x}|x|^{1/\epsilon})-\delta (t_0 -r_0 -t) . $$
We show that, for all $\epsilon$ sufficiently small, depending on $\psi$, $r_0$, $x_0$, and $t_0$, $\psi^\epsilon$
is a sub-solution to \eqref{eq-v-eps} in the cylinder $D_{r_0}(x_0,t_0)$.\\
\\
Let $\phi$ be a smooth test function such that $\psi^\epsilon - \phi$ assumes a
maximum value of $0$ at $(\bar{x},\bar{t}) \in D_{r_0}(x_0,t_0)$ with $|\bar{x}|>1$.
Since
\begin{equation}
1- \frac{\exp(\phi(\eta^\epsilon(\bar{x},y),\bar{t})/\epsilon)}{\exp(\phi(\bar{x},\bar{t})/\epsilon)}
\leq 1 - \frac{\exp(\psi^\epsilon(\eta^\epsilon(\bar{x},y),\bar{t})/\epsilon)}
{\exp(\psi^\epsilon(\bar{x},\bar{t})/\epsilon)} ,
\label{subsol-pre-1}
\end{equation}
in view of  \eqref{theta-ident-2}$(v)$ we have 
\begin{equation}
\phi_t(\bar{x},\bar{t}) \leq \psi_t^\epsilon(\bar{x},\bar{t}) =
\theta_t(\bar{x},\bar{t}) + \delta < |\lambda_1|-A_{k+1}-\sigma/3 .
\label{subsol-pre-2}
\end{equation}
Using \eqref{subsol-pre-1} and \eqref{subsol-pre-2} we find that, for any $R>0$, at the point $(\bar{x},\bar{t})$
\begin{align}
\phi_t &+ \int_{B_R(\bar{x})} \left( 1-
\frac{\exp(\phi(\eta^\epsilon(\bar{x},y),\bar{t})/\epsilon)}{\exp(\phi/\epsilon)}
\right) K dy + \int_{B_R(\bar{x})^C} \left( 1- \frac{\exp(\psi^\epsilon(\eta^\epsilon(\bar{x},y),\bar{t})/\epsilon)}
{\exp(\psi^\epsilon/\epsilon)} \right) K dy \nonumber \\
& \leq \theta_t + \frac{\sigma}{3} + \int \left( 1-
\frac{\exp(\psi^\epsilon(\eta^\epsilon(\bar{x},y),\bar{t})/\epsilon)}
{\exp(\psi^\epsilon/\epsilon)} \right) K dy
< |\lambda_1| - A_{k+1} - \frac{\sigma}{3} +
\frac{L^\alpha[e^{\psi^\epsilon}]}{e^{\psi^\epsilon}}
\nonumber \\
& = -A_{k+1}-\frac{\sigma}{3} +
\mu(\hat{\bar{x}}|\bar{x}|^{1/\epsilon}) +
\mathcal{I}_{\epsilon,\alpha} +
\int_{|y|\geq r} G_\epsilon(\bar{x},y) \left(1-
\frac{\exp(\theta(\eta^\epsilon(\bar{x},y),\bar{t})/\epsilon)}
{\exp(\theta/\epsilon)} \right) K dy ,
\label{subsol-main}
\end{align}
where  $r = r(\epsilon) = \epsilon |\bar{x}|^{1/\epsilon}$ and 
\begin{align*}
\mathcal{I}_{\epsilon,\alpha} := \left\{
\begin{array}{ll}
\medskip
\int_{|y| < r} G_\epsilon(\bar{x},y) \left(1- \frac{\exp(\theta(\eta^\epsilon(\bar{x},y),
\bar{t})/\epsilon)}{\exp(\theta(\bar{x},\bar{t})/\epsilon)} \right) K dy , & \text{ if } \ \alpha < 1, \\
\medskip
\frac{1}{2} \int_{|y| < r} (H_\epsilon(\theta,\bar{x},y)+H_\epsilon(\theta,\bar{x},-y)) K dy,
& \text{ if } \ \alpha \geq 1.
\end{array}
\right.
\end{align*}
\\
Compared to the proof of Proposition \ref{prop-induction-1}, the situation here
is much simpler. For instance, the direction of the inequality in
\eqref{subsolution} means we only need to bound the last term in \eqref{subsol-main}
from above. Since the integrand is bounded above uniformly by $K \| G_\epsilon \|$,
we conclude from \eqref{takis4} that
$$ \int_{|y|\geq r} G_\epsilon(\bar{x},y) \left(1-
\frac{\exp(\psi^\epsilon(\eta^\epsilon(\bar{x},y),\bar{t})/\epsilon)}
{\exp(\psi^\epsilon(\bar{x},\bar{t})/\epsilon)} \right) K dy \lesssim
\int_{|y|\geq r} \frac{dy}{|y|^{d+\alpha}} \lesssim r^{-\alpha} . $$
As for the short range, the formula for $\mathcal{I}_{\epsilon,\alpha}$ above is identical to
the one given after \eqref{supersol-rhs} of the previous section. As such, an identical
analysis of $\mathcal{I}_{\epsilon,\alpha}$ as done in the proof of Proposition \ref{prop-induction-1}
using Taylor's theorem yields that \eqref{I-1} and
\eqref{I-2} still hold, respectively when $\alpha \in [1,2)$ and $\alpha \in (0,1)$.\\
\\
To prove that $\underset{\epsilon \rightarrow 0} \lim \ |\mathcal{I}_{\epsilon,\alpha}| = 0$, the only
remaining issue is to bound the size of $\sup_{|z|<r} h(z)$, where
$$ h(y) = \frac{\exp(\theta(\eta^\epsilon(\bar{x},y),\bar{t})/\epsilon)}
{\exp(\theta(\bar{x},\bar{t})/\epsilon)} . $$
Using \eqref{psi-is-good-lower-bound}, \eqref{theta-ident-2}$(iii)$, \eqref{theta-ident-2}$(iv)$,
and \eqref{b-ratio} we have, for all $|y|<r$,
\begin{align}
h(y) &\leq \frac{\exp (-A_{k+1} \bar{t}/\epsilon +
\min (0,|\lambda_1| \bar{t}/\epsilon -(d+\alpha) \log (|\hat{\bar{x}}
|\bar{x}|^{1/\epsilon} - y|))}
{\exp (-A_k \bar{t}/\epsilon +
\min (0, |\lambda_1| \bar{t}/\epsilon-(d+\alpha) \log |\bar{x}|/\epsilon))}
e^{\delta (\bar{t}-t_0 +3 r_0)/\epsilon}e^{g(\hat{\bar{x}}|\bar{x}|^{1/\epsilon})} \nonumber \\
&\lesssim
\frac{1+e^{-|\lambda_1|\bar{t}/\epsilon}|\bar{x}|^{(d+\alpha)/\epsilon}}
{1+e^{-|\lambda_1|\bar{t}/\epsilon}|\hat{\bar{x}}|\bar{x}|^{1/\epsilon}-y|^{d+\alpha}}
e^{(A_k - A_{k+1}+\delta) \bar{t}/\epsilon}
\lesssim \left( e^{(A_k-A_{k+1}+\delta) \bar{t}}
\min(1, e^{-|\lambda_1| \bar{t}} |\bar{x}|^{d+\alpha}) \right)^{1/\epsilon} .
\label{h-bound-2}
\end{align}
It then follows from  \eqref{I-1}, \eqref{I-2}, and \eqref{h-bound-2} that, for $\alpha \in (0,2)$ and all $\epsilon$ sufficiently small,
$$ |\mathcal{I}_{\epsilon,\alpha}| \lesssim O(\epsilon) + |\bar{x}|^{-\alpha/\epsilon} + \mathcal{H}_\epsilon , $$
with
$$ \mathcal{H}_\epsilon :=
e^{(A_k-A_{k+1}+\delta)\bar{t}/\epsilon} |\bar{x}|^{-1/\epsilon}\min(1,
e^{-|\lambda_1| \bar{t}} |\bar{x}|^{d+\alpha})^{1/\epsilon} . $$
If
$|\bar{x}|^{d+\alpha} \geq e^{|\lambda_1|\bar{t}/2}$, then
$$ \mathcal{H}_\epsilon \leq \left( e^{(A_k-A_{k+1}+\delta)\bar{t}}|\bar{x}|^{-1} \right)^{1/\epsilon}
\leq \exp \left( \left(A_k-A_{k+1}+\delta-\frac{|\lambda_1|}{2(d+\alpha)} \right)
\frac{\bar{t}}{\epsilon} \right) = c_1^{1/\epsilon} , $$
and otherwise 
$$ \mathcal{H}_\epsilon \leq \left( e^{(A_k-A_{k+1}+\delta)\bar{t}}
e^{-|\lambda_1| \bar{t}} |\bar{x}|^{d+\alpha} \right)^{1/\epsilon} \leq
\exp \left( \left( A_k - A_{k+1}+\delta - \frac{|\lambda_1|}{2} \right)
\frac{\bar{t}}{\epsilon} \right) = c_2^{1/\epsilon} , $$
where $0<c_2<c_1<1$ and both constants are independent of $\epsilon$.\\
\\
Hence $\underset{\epsilon \to 0}\lim \ \mathcal{H}_\epsilon =0$, and,
therefore, for all $\epsilon$ sufficiently small and independently of $\phi$, \eqref{subsol-main} becomes
$$ \phi_t (\bar{x},\bar{t})+\int_{B_R(\bar{x})} \left(1- \frac{\exp(\phi(\eta^\epsilon
(\bar{x},y),\bar{t})/\epsilon)}{\exp(\phi(\bar{x},\bar{t})/\epsilon)} \right) K dy
+\int_{B_R(\bar{x})^c} \left(1-h(y) \right) K dy
\leq \mu(\hat{\bar{x}}|\bar{x}|^{1/\epsilon}) - \frac{\sigma}{3} - A_{k+1} , $$
that is $\psi^\epsilon$ is a sub-solution.\\
\\
Next we derive a contradiction by comparing $\psi^\epsilon$ with $v^\epsilon$. Let
$K:=\overline{B}_{R_K}(0) \times [0,t_0+r_0+1]$ with
$$R_K := 3 \exp \left( \frac{(A_0+|\lambda_1|+4 \delta)(t_0+r_0)}{d+\alpha} \right) (|x_0|+r_0)+1 . $$
It follows from \eqref{loc-unif-limits} and \eqref{psi-is-good-lower-bound} that, for all $\epsilon$ sufficiently
small and all $(x,t) \in K$,
$$ v^\epsilon > v_* - \delta r_0 / 2 \geq \psi - \delta r_0 / 2
\geq \min(0,|\lambda_1|t-(d+\alpha)\log(|x|))-A_k t-\delta r_0/2 . $$
Let $\bar{v}^\epsilon$ be such that
\begin{equation}\label{v-eps-bar-2}
\bar{v}^\epsilon = v^\epsilon \ \text{ in } \ \overline{B}_{R_K-1}(0) \times [0,t_0+r_0], \  \bar{v}^\epsilon \geq \psi \ \text{ in } \ K^C, \  \text{and} \ \bar{v}^\epsilon \geq v^\epsilon \ \text{ in } \ \RR^d \times [0,\infty). 
\end{equation}
Then, in light of \eqref{theta-ident-2}$(ii)$, we see that, for all $(x,t) \in D_{r_0}(x_0,t_0)^C$ and $\epsilon$ sufficiently
small that $\epsilon g \leq \delta r_0 / 2$,
$$ \psi^\epsilon = \theta+\epsilon g(\hat{x}|x|^{1/\epsilon}) - \delta (t_0-r_0-t)
\leq \psi - \delta r_0 / 2 < \bar{v}^\epsilon (x,t) . $$
If $\bar{t} \leq t_0+r_0$ is the first time such that $\underset{x\in \RR^d} \max(\psi^\epsilon(\cdot, \bar{t})-\bar{v}^\epsilon
(\cdot, \bar{t})) = 0$, let $\bar{x} \in B_{r_0}(x_0)$ be a point in space where this maximum is achieved.\\
\\
Then $\bar{v}^\epsilon$ becomes
an admissible test function and must satisfy 
$$ \bar{v}_t^\epsilon (\bar{x},\bar{t})+\int_{B_R(\bar{x})} \left(1- \frac{\exp(\bar{v}^\epsilon(\eta^\epsilon
(\bar{x},y),\bar{t})/\epsilon)}{\exp(\bar{v}^\epsilon(\bar{x},\bar{t})/\epsilon)} \right) K dy
+\int_{B_R(\bar{x})^c} \left(1-h(y) \right) K dy
\leq \mu(\hat{\bar{x}}|\bar{x}|^{1/\epsilon}) - \frac{\sigma}{3} - A_{k+1} . $$
On the other hand,
$$ \bar{v}_t^\epsilon (\bar{x},\bar{t})+\int \left(1-
\frac{\exp(\bar{v}^\epsilon(\eta^\epsilon (\bar{x},y),\bar{t})/\epsilon)}
{\exp(\bar{v}^\epsilon(\bar{x},\bar{t})/\epsilon)} \right) K dy +
\mathcal{J}_\epsilon = \mu(\hat{\bar{x}}|\bar{x}|^{1/\epsilon}) - E e^{-\bar{v}^\epsilon
(\bar{x},\bar{t})/\epsilon} , $$
where
$$ \mathcal{J}_\epsilon := \int \frac{\exp(\bar{v}^\epsilon(\eta^\epsilon (\bar{x},y),\bar{t})/\epsilon)
-\exp(v^\epsilon(\eta^\epsilon (\bar{x},y),\bar{t})/\epsilon)}
{\exp(\bar{v}^\epsilon(\bar{x},\bar{t})/\epsilon)} K dy . $$
Subtracting the last two equations yields, in light of \eqref{E-assumption},
\begin{equation}
\int_{B_R(\bar{x})^C} \frac{\exp(\bar{v}^\epsilon(\eta^\epsilon(\bar{x},y),\bar{t})/\epsilon)
-\exp(\psi^\epsilon(\eta^\epsilon(\bar{x},y),\bar{t})/\epsilon)}{\exp(\bar{v}^\epsilon(\bar{x},\bar{t})/\epsilon)}
K dy - \mathcal{J}_\epsilon \leq \bar{m} e^{v^\epsilon(\bar{x},\bar{t})/\epsilon} - \frac{\sigma}{3}-A_{k+1} .
\label{breakpoint}
\end{equation}
\\
Note that the first term in the left hand side of \eqref{breakpoint} is nonnegative, while, as before, 
$$ \lbrace{ y: \ |\eta^\epsilon(\bar{x},y)| > R_K - 1 \rbrace} \subset S_\epsilon :=
\lbrace{ y: \ |y| > 2 |\bar{x}|^{1/\epsilon} e^{\frac{(A_0+|\lambda_1|+4 \delta)\bar{t}}
{(d+\alpha) \epsilon}} \rbrace} \ \  \text{and} \ \  v^\epsilon = \bar{v}^\epsilon \ \text{ on} \  S_\epsilon^C.$$
It follows from \eqref{psi-is-good-lower-bound} that
\begin{align}
\mathcal{J}_\epsilon &\leq \int_{S_\epsilon} \frac{\exp(\bar{v}^\epsilon(\eta^\epsilon(\bar{x},y),\bar{t})/\epsilon)}
{\exp(\bar{v}^\epsilon(\bar{x},\bar{t})/\epsilon)} K dy \nonumber \\
&\lesssim \int_{S_\epsilon} \frac{\exp((\min(0,|\lambda_1|\bar{t}
-(d+\alpha) \log (|\hat{\bar{x}}|\bar{x}|^{1/\epsilon}-y|^\epsilon)) - A_{k+1} \bar{t} /\epsilon)}
{\exp((\min(0,|\lambda_1| \bar{t}-(d+\alpha) \log (|\bar{x}|))-A_k \bar{t})/\epsilon)}
\frac{e^{4\delta r_0 /\epsilon}}{|y|^{d+\alpha}}dy \nonumber \\
&\lesssim e^{(A_k-A_{k+1}+4\delta r_0)/\epsilon} \int_{S_\epsilon}
\frac{1 + e^{-|\lambda_1| \bar{t}/\epsilon} |\bar{x}|^{(d+\alpha)/\epsilon}}
{e^{-|\lambda_1| \bar{t}/\epsilon} |\hat{\bar{x}}|\bar{x}|^{1/\epsilon}-y|^{d+\alpha}}
\frac{dy}{|y|^{d+\alpha}} \lesssim \int_{S_\epsilon} \frac{dy}{|y|^{d+\alpha}}
\lesssim 2^{-\alpha} |\bar{x}|^{-\alpha/\epsilon} e^{-\frac{\alpha (A_0+|\lambda_1|+4 \delta)\bar{t}}
{(d+\alpha)\epsilon}} . \nonumber
\end{align}
Hence, as $\epsilon \rightarrow 0$, $\mathcal{J}_\epsilon$ becomes nonpositive. Then, for all $\epsilon$ sufficiently small,
\begin{equation}
\frac{\sigma}{3} \leq \bar{m} e^{v^\epsilon(\bar{x},\bar{t})/\epsilon} = \bar{m} e^{\psi^\epsilon
(\bar{x},\bar{t})/\epsilon} .
\label{breakpoint-2}
\end{equation}
The left hand side of \eqref{breakpoint-2} is fixed and strictly positive.
However, in view of  \eqref{theta-ident-2}$(iv)$ and the definition of $\psi^\epsilon$,
$$ \psi^\epsilon(\bar{x},\bar{t}) \leq -A_{k+1} \bar{t} + \delta r_0 -\delta (t_0-r_0-\bar{t})
+ \epsilon g(\hat{\bar{x}}|\bar{x}|^{1/\epsilon}),$$
so that
\begin{equation}
e^{\psi^\epsilon(\bar{x},\bar{t})/\epsilon} \lesssim e^{-A_{k+1} (t_0-r_0)/\epsilon}
e^{3 \delta r_0 / \epsilon} .
\label{breakpoint-3}
\end{equation}
In light  of \eqref{delta-2}, the right hand side of \eqref{breakpoint-3} becomes arbitrarily small
as $\epsilon \rightarrow 0$, a contradiction to \eqref{breakpoint-2}.
It follows that there is no such $(\bar{x},\bar{t})$ in $D_{r_0}(x_0,t_0)$, and we conclude that,
for all $(x,t) \in D_{r_0/2}(x_0,t_0)$ and all $\epsilon$ sufficiently small,
$$ \psi^\epsilon = \theta+\epsilon g(\hat{x}|x|^{1/\epsilon})+\delta r_0/2 < v^\epsilon . $$
Since $\theta=\psi$ on this set by \eqref{theta-ident-2}$(i)$,
$$ \psi(x_0,t_0) \leq v_*(x_0,t_0) - \delta r_0 / 2 , $$
contradicting \eqref{a-bound}. This shows that \eqref{supersolution-v-lower-not} cannot hold,
thus proving the variational inequality \eqref{improved-bounds-lower} for all $(x,t) \in B_1^C \times (0,\infty).$ \\
\\
Next we establish an improved lower bound in $\overline B_1  \times (0,\infty)$.
Note that, since we have already proved Proposition~\ref{prop-induction-1} independently of the rest of this section, it follows that
\begin{equation}
v^* \leq \min(0,|\lambda_1|t - (d+\alpha) \log(|x|)) \ \text{ for all } \ (x,t) \in \RR^d \times [0,\infty) .
\label{already-have-this}
\end{equation}
Indeed, since the proof of Proposition \ref{prop-induction-1} never used Proposition \ref{prop-induction-2}
or the lower bound of \eqref{rough-v-bounds} for any $k$, the inductive argument of Section 3 shows
that \eqref{already-have-this} holds independently.\\
\\
Also, since the previous argument showed that \eqref{improved-bounds-lower} holds in  $B_1^C\times (0,\infty),$
we use Lemma \ref{lem-induction-complete} (with $U = \overline{B}_1(0)^C$) to conclude that 
\begin{equation}\label{inductive-lower-bound-semi}
\min(0,|\lambda_1|t-(d+\alpha) \log(|x|))-A_{k+1} t \leq v_* \ \text{ for all } \ (x,t) \in
B_1^C\times (0,\infty).
\end{equation}
Assuming the next lemma, we have now concluded the proof of Proposition \ref{prop-induction-2}.
\end{proof}

\begin{lemma}
Assume \eqref{already-have-this} and \eqref{inductive-lower-bound-semi}.
Then, for all $(x,t)\in \overline B_1 \times (0,\infty)$, we have
$ v_* \geq -A_{k+1}t$.
\label{lem-inner-lower-bound}
\end{lemma}
\begin{proof}
Fix $T>0$,  let $R = e^{2|\lambda_1|T/(d+\alpha)}$ and  
$$ \overline{S}_{R,T} := \lbrace{(x,t): \  1 \leq |x| \leq R , \ T/2 \leq t \leq T \rbrace} . $$
Then, for any $\nu>0$, all $(x,t) \in \overline{S}_{R,T}$, and
all $\epsilon$ sufficiently small, it follows from \eqref{loc-unif-limits} that
\begin{equation}
v^\epsilon > \min(0,|\lambda_1|t-(d+\alpha)\log(|x|))-A_{k+1}t-\nu ,
\label{c1}
\end{equation}
which implies
$$u(\hat{x}|x|^{1/\epsilon},t/\epsilon) = u^\epsilon(x,t) = e^{v^\epsilon(x,t)/\epsilon} >
\frac{C_1e^{-A_{k+1}t/\epsilon}e^{-\nu/\epsilon}}{1+e^{-|\lambda_1|t/\epsilon}
|x|^{(d+\alpha)/\epsilon}} . $$
Hence,
\begin{equation}
u(y,s) > \frac{C_1e^{-A_{k+1}s-\nu/\epsilon}}{1+e^{-|\lambda_1|s}|y|^{d+\alpha}} \
\text{ on } \ \lbrace{(y,s): \ 1 \leq |y| \leq R^{1/\epsilon} ,
\ T/(2\epsilon) \leq s \leq T/\epsilon \rbrace} .
\label{c2}
\end{equation}
Keeping in mind that the solution $u$ is bounded above by a fixed constant for all time, it follows that
there exists $C_0>0$ such that, uniformly in $x$,
\begin{equation}
f(x,u) \geq -C_0 u .
\label{poor-mans-bound}
\end{equation}
Assume that $v^\epsilon$ achieves its minimum on the set $\overline{B}_1(0)
\times [T/2,T]$ at $(x_0,t_0)$ and that, for some $\sigma > 0$, $v^\epsilon(x_0,t_0) = -A_{k+1}t_0-\sigma$.
Then $u$ achieves its minimum value on $\overline{B}_1(0) \times [T/(2\epsilon),T/\epsilon]$
at the point $(y_0,s_0) = (\hat{x_0}|x_0|^{1/\epsilon}, t_0/\epsilon)$, and 
in view of  \eqref{main-eq} and \eqref{poor-mans-bound},
$$ u_t = f(y_0,u) - L^\alpha[u]
\geq -C_0 u + \int_{|y_0+z| > 1} (u(y_0+z,s_0)-u(y_0,s_0))K(y_0,z)dz . $$
Note that the inner range $\overline B_1(-y_0)$ of integration can be safely omitted from the inequality because,
by assumption, $u(w,s_0)-u(y_0,s_0) \geq 0$ for $|w| \leq 1$.\\
\\
Using \eqref{c1}, \eqref{c2}, and the size of $R$, we get that
$$ u_t(y_0,s_0) \gtrsim -C_0e^{-A_{k+1} s_0-\sigma/\epsilon} +
e^{-A_{k+1}s_0-\nu/\epsilon} I_\epsilon
-e^{-A_{k+1}s_0-\sigma/\epsilon} \int_{|y_0+z|^{d+\alpha}
\geq e^{|\lambda_1| s_0}} \frac{dz}{|z|^{d+\alpha}} , $$
where
$$ I_\epsilon := \int_{1<|y_0+z|^{d+\alpha}< e^{|\lambda_1|s_0}}
\left( \frac{C_1}{1+e^{-|\lambda_1|s_0}|y_0+z|^{d+\alpha}} - e^{(\nu-\sigma)/\epsilon} \right)
\frac{dz}{|z|^{d+\alpha}} . $$
Choosing $\nu < \sigma / 2$ and
$\epsilon$ sufficiently small depending on $\sigma$ and $C_1$, we have
$$e^{(\nu-\sigma)/\epsilon} < e^{-\sigma/(2\epsilon)} < C_1/4 . $$
Moreover, since $|y_0|\leq 1$, we have $|y_0+z| \leq |z|+1$, and,  therefore,
$$ \lbrace{2 < |z|^{d+\alpha} < e^{|\lambda_1|s_0}/2 \rbrace}
\subseteq \lbrace{1 < |y_0+z|^{d+\alpha} < e^{|\lambda_1|s_0} \rbrace} \ 
\text{and} \ 
 \lbrace{ |y_0 + z|^{d+\alpha} \geq e^{|\lambda_1|s_0} \rbrace}
\subseteq \lbrace{|z|^{d+\alpha} \geq e^{|\lambda_1| s_0}/2 \rbrace}.$$
Lastly, because  $s_0 > T/(2\epsilon)$, we know that
$e^{|\lambda_1|s_0} \rightarrow \infty$ as $\epsilon \rightarrow 0$.
Then, for all $\epsilon$ sufficiently small, there exists a fixed $C_1' > 0$ such that
$$ I_\epsilon \gtrsim \int_{2<|z|^{d+\alpha}< e^{|\lambda_1|s_0}}
\frac{dz}{(1+e^{-|\lambda_1|s_0}|z|^{d+\alpha})|z|^{d+\alpha}}
\gtrsim C_1' . $$
Therefore, for all $\epsilon$ sufficiently small,
$$ u_t(y_0,s_0) \gtrsim e^{-A_{k+1} s_0 - \sigma/(2\epsilon)} \left(
-C_0 e^{-\sigma/(2\epsilon)} + C_1' - C_2 e^{-\sigma/(2\epsilon)} e^{-\alpha |\lambda_1| s_0} \right), $$
that is, the time derivative of $u$ is positive at its minimum
value on $\overline{B}_1(0) \times [T/(2\epsilon),T/\epsilon]$. This can only happen if $s_0 = T/(2\epsilon)$.
Since $\sigma > 0$ was arbitrary, we conclude that, if $v^\epsilon$ has a minimum on $\overline{B}_1(0)
\times [T/2,T]$ that lies below $-A_{k+1} t$, then this minimum must occur at time $T/2$.\\
\\
It follows that the same must hold true for $v_*$,
which  is lower semicontinuous, so it always has a minimum on compact sets. 
Also recall that $T>0$ was arbitrary.
If there is any time $t_0$ such that
$$ \min_{|x|\leq 1} v_*(x,t_0) < -A_{k+1}t_0 , $$
then we apply the previous  conclusion for the set $\bar{B_1}(0) \times [t_0/2,t_0]$ and obtain that
$$ \min_{|x|\leq 1} v_*(x,t_0/2) < -A_{k+1}t_0 < -A_{k+1}t_0/2 . $$
Repeating this argument, we find that for all $m \geq 0$
$$ \min_{|x|\leq 1} v_*(x,t_0 2^{-m}) < -A_{k+1}t_0 . $$
This is in clear contradiction to \eqref{rough-v-bounds}.
Therefore, $v_* \geq -A_{k+1} t$ for all $|x| \leq 1$ and $t>0$.\\
\end{proof}

\section{Proof of Theorem \ref{main-corollary}}
\begin{proof}[The proof of Theorem \ref{main-corollary}]
The first claim of \eqref{limit-for-u} follows
immediately from the Hopf-Cole transformation and \eqref{homogenized}. That is, if $|x|^{d+\alpha}
> e^{|\lambda_1|t}$, then $v^\epsilon$ converges locally uniformly to $|\lambda_1| t - (d+\alpha) \log (|x|) < 0$,
and so $u^\epsilon = e^{v^\epsilon /\epsilon}$ converges locally uniformly to zero as $\epsilon \rightarrow 0$.\\
\\
Let $u$ be the solution to \eqref{main-eq} and assume the conclusions of Theorem \ref{main-theorem}.
Since both $u$ and $u^+$ are positive,  we write
$$ w = u / u^+ , $$
and note that, throughout this section, $w$ will always be in unscaled coordinates. Then, by \eqref{main-eq},
$$ u^+ w_t + \int (u^+(x) w(x,t) - u^+(x+y) w(x+y,t)) K(x,y) dy = f(x,u^+ w), $$
and
$$ u^+ w_t + \int (w(x,t)-w(x+y,t)) u^+(x+y) K(x,y) dy = f(x,u^+ w) - f(x, u^+) w, $$
that is,
\begin{equation}
w_t + \int (w(x,t)-w(x+y,t)) \frac{u^+(x+y)}{u^+(x)} K(x,y) dy
= w \left( \frac{f(x,u^+w)}{u^+w} - \frac{f(x,u^+)}{u^+} \right) .
\label{u-over-steady}
\end{equation}
We write
$$ \overline{K}(x,y) = \frac{u^+(x+y)}{u^+(x)} K(x,y) \ \text{ and } \
N(x,w) = \frac{f(x,u^+w)}{u^+w} - \frac{f(x,u^+)}{u^+} , $$
and note that $\overline{K}$ also satisfies \eqref{takis3}, \eqref{takis4} with different constants.\\
\\
It follows from  \eqref{takis2} that $N(x,w)$ is decreasing in $w$,
negative when $w > 1$, and positive when $0 < w < 1$.
Since $w$ decays at infinity like $u$, it follows that, for a given $t \geq 0$, $w(\cdot, t)$ achieves
its maximum value at some  $\bar{x}\in \RR^d$.\\
\\
Then, at $(\bar{x},t)$,  \eqref{u-over-steady} implies that
$$ w_t \leq w \left( \frac{f(\bar{x},u^+w}{u^+w}
- \frac{f(\bar{x},u^+)}{u^+} \right) , $$
and, for any $\sigma > 0$, if
$$ \lim_{t_0 \rightarrow \infty} \sup_{(x,t) \in \RR^d \times [t_0,\infty)} w(x,t) > 1+ \sigma , $$
then there exist $c_\sigma, t_\sigma>0$ depending only on $f$, $C_0$ and the global upper bound for $w$,
such that, for all $(x,t) \in \RR^d \times [t_\sigma,\infty)$,
$$ w(x,t) < C_0 + (1+\sigma) c_\sigma (t_\sigma-t),$$
that is, $C_0 + (1+\sigma) c_\sigma (t_\sigma-t)$ is a barrier from above for \eqref{u-over-steady}.
This is an obvious contradiction to \eqref{rough-bounds}. Since $\sigma > 0$ was arbitrary, we conclude that
\begin{equation}
\lim_{t \rightarrow \infty} \sup_{x \in \RR^d} w(x,t) \leq 1 .
\label{u-over-steady-sup}
\end{equation}
To complete the proof of \eqref{limit-for-u}, we essentially need the analog of \eqref{u-over-steady-sup} for the
infimum of $w$. However, $w$ decays to zero at infinity and lacks a global minimum. Instead, for
$\xi > 1$ and $M>1$, we consider the function
$$ h(x,t) := M+ e^{-|\lambda_1|t} |x|^{\xi (d+\alpha)} . $$
Writing $W = w h$, it follows from \eqref{rough-bounds} that $W(\cdot,t)$ must,
for every $t > 0$, achieve a global minimum at some point $\bar{x} \in \RR^d$. Moreover, since $\xi > 1$, for
all $t$ sufficiently large depending on $\xi$, $|\bar{x}|^{\xi(d+\alpha)} <
e^{|\lambda_1| t}$. It follows from \eqref{u-over-steady} that
\begin{align}
W_t &= -|\lambda_1| e^{-|\lambda_1|t} |x|^{\xi(d+\alpha)} w
+ h \left( w N(x,w) + \int (w(x+y,t)-w(x,t)) \overline{K}(x,y) dy \right) \nonumber \\
&= W \left(N(x,w) -
\frac{|\lambda_1| e^{-|\lambda_1| t} |x|^{\xi(d+\alpha)}}
{M+e^{-|\lambda_1| t} |x|^{\xi(d+\alpha)}} \right)
+ \int (W(x+y,t)-W(x,t)) \overline{K}(x,y) dy + \bar{I} ,
\label{W-eq}
\end{align}
with
\begin{align}
\bar{I} &= \int (h (x,t) - h (x+y,t)) w(x+y,t) \overline{K}(x,y) dy \nonumber \\
&= e^{-|\lambda_1| t} \int (|x|^{\xi(d+\alpha)}-|x+y|^{\xi(d+\alpha)})
w(x+y,t) \overline{K}(x,y) dy . \nonumber
\end{align}
For $T>0$ and for $B_0$ given by \eqref{rough-bounds}, let
$$ R = \exp \left( \frac{3T (B_0-|\lambda_1|)}{\alpha - (\xi-1)(d+\alpha)} \right),$$
and choose $\xi$ sufficiently close to $1$ that the exponent above is positive.\\
\\
Fix $\nu > 0$. Then, for all $\epsilon$ sufficiently small, Theorem \ref{main-theorem}
implies that
$$ v^\epsilon < \min(0,|\lambda_1|t-(d+\alpha) \log (|x|)) + \nu  \ \text{ in } \
\lbrace{ (x,t) \ : \ |x| \leq R , \ T \leq t \leq 2T \rbrace},$$
that is, for some $C_1>0$ independent of $x$, $t$, or $\epsilon$,
$$ u(\hat{x}|x|^{1/\epsilon},t/\epsilon) < \frac{C_1 e^{\nu/\epsilon}}{1+e^{-|\lambda_1|t/\epsilon}
|x|^{(d+\alpha)/\epsilon}} . $$
Recall that $u^+$ is bounded uniflormly from above and below. Hence, for a different constant $C_1' > 0$,
\begin{equation}
w(x,t) < \frac{C_1' e^{\nu/\epsilon}}{1+e^{-|\lambda_1|t}|x|^{d+\alpha}}  \ \ \text{ in } \ \
\lbrace{ (x,t) \ : \ |x| \leq R^{1/\epsilon} , \ T/\epsilon \leq t \leq 2T/\epsilon \rbrace} .
\label{upper-on-u-over-steady}
\end{equation}
Looking at the minimum $(\bar{x},\bar{t})$ of $w$ on $\RR^d \times [T/\epsilon,
2T/\epsilon]$, \eqref{W-eq} yields
\begin{equation}
W_t \geq \left( N(\bar{x},w) - |\lambda_1|/M \right) W + \bar{I} .
\label{W-min}
\end{equation}
We now show that, for an appropriate choice of parameters, $\underset{\epsilon \rightarrow 0}\liminf \ \bar{I} \geq 0$.
To do this, we split the domain of integration in the formula for $\bar{I}$ into three parts, that is,
\begin{equation}
\bar{I} \geq \bar{I}_1 + \bar{I}_2 + \bar{I}_3 ,
\label{bar-I-decomp}
\end{equation}
where, for $\rho := \epsilon e^{|\lambda_1| \bar{t}}$,
\begin{align}
\bar{I}_1 &:= -\left| \int_{|y|^{\xi(d+\alpha)} \leq \rho}
(h (\bar{x},\bar{t}) - h (\bar{x}+y,\bar{t})) w(\bar{x}+y,\bar{t}) \bar{K}(\bar{x},y) dy \right| , \nonumber \\
\bar{I}_2 &:= \int_{|y|^{\xi(d+\alpha)} > \rho , \ |\bar{x}+y| \leq R^{1/\epsilon}}
(h (\bar{x},\bar{t}) - h (\bar{x}+y,\bar{t})) w(\bar{x}+y,\bar{t}) \bar{K}(\bar{x},y) dy. \nonumber \\
\bar{I}_3 &:= \int_{|\bar{x}+y|>R^{1/\epsilon}}
(h (\bar{x},\bar{t}) - h (\bar{x}+y,\bar{t})) w(\bar{x}+y,\bar{t}) \bar{K}(\bar{x},y) dy. \nonumber
\end{align}
We proceed to estimate $\bar{I}_1$. This requires two slightly different treatments depending
on the value of $\alpha$.\\
\\
When $\alpha \in [1,2)$,  a change of variables yields
\begin{align}
& \int_{|y|\leq \rho} (h(\bar{x},\bar{t})-h(\bar{x}+y,\bar{t}))w(\bar{x}+y,\bar{t}) \overline{K}(\bar{x},y) dy
\nonumber \\
& \hspace{1.5cm} = \frac{1}{2} \int_{|y|\leq \rho} (2h(\bar{x},\bar{t})-h(\bar{x}+y,\bar{t})-h(\bar{x}-y,\bar{t}))
w(\bar{x}+y,\bar{t}) \overline{K}(\bar{x},y) dy
\nonumber \\
& \hspace{2.5cm} + \frac{1}{2} \int_{|y|\leq \rho} (h(\bar{x},\bar{t})-h(\bar{x}-y,\bar{t}))
(w(\bar{x}+y,\bar{t})-w(\bar{x}-y,\bar{t})) \overline{K}(\bar{x},y) dy.
\label{I1-prelim}
\end{align}
Recall that $w$ and $Dw$ are bounded uniformly, and observe that, if $|x|^{\xi (d+\alpha)} < e^{|\lambda_1| t}$, then
\begin{equation}
|Dh(x,t)| \lesssim e^{-\frac{|\lambda_1| t}{\xi (d+\alpha)}} \ \text{ and } \
|D^2 h(x,t)| \lesssim e^{-\frac{2|\lambda_1| t}{\xi (d+\alpha)}} .
\label{Dh-bound}
\end{equation}
It follows from  \eqref{I1-prelim}, \eqref{Dh-bound},
and Taylor's Theorem that
\begin{align}
\bar{I}_1 &= -\left| \int_{|y|^{\xi(d+\alpha)} \leq \rho}
(h(\bar{x},\bar{t})-h(\bar{x}+y,\bar{t}))w(\bar{x}+y,\bar{t}) \overline{K}(\bar{x},y) dy
\right| \nonumber \\
& \gtrsim - \int_{|y|^{\xi(d+\alpha)} \leq \rho}
\left( \sup_{|z|^{\xi(d+\alpha)} < e^{|\lambda_1|\bar{t}}}
\left( |D^2 h(z,\bar{t})| + |D h(z,\bar{t})| |D w(z,\bar{t})| \right) \right)
|y|^2 \overline{K}(\bar{x},y) dy \nonumber \\
&\gtrsim - e^{-\frac{|\lambda_1| \bar{t}}{\xi (d+\alpha)}}
\int_{|y|^{\xi(d+\alpha)} \leq \rho} |y|^{2-d-\alpha} dy
\gtrsim -e^{-\frac{|\lambda_1| \bar{t}}{\xi(d+\alpha)}}
\left(\epsilon e^{|\lambda_1|\bar{t}} \right)^{\frac{2-\alpha}{\xi(d+\alpha)}} .
\label{I-bar-short}
\end{align}
Since $\bar{t} \in [T/\epsilon,2T/\epsilon]$ and $\alpha \in [1,2)$,
it follows that the lower bound converges to zero as $\epsilon \rightarrow 0$. \\
\\
When $\alpha \in (0,1)$,  the argument  follows immediately from the calculations above.
Indeed, using the first bound of \eqref{Dh-bound} yields
\begin{equation}
\bar{I}_1 \gtrsim - \int_{|y| \leq \rho}
\left( \sup_{|z|^{d+\alpha} \leq e^{|\lambda_1|\bar{t}}} |Dh(z,\bar{t})| \right)
|y| \overline{K}(\bar{x},y) dy \gtrsim -e^{-\frac{|\lambda_1|\bar{t}}{d+\alpha}}
\left( \epsilon e^{|\lambda_1|\bar{t}} \right)^{\frac{1-\alpha}{\xi(d+\alpha)}} .
\label{I-bar-short-2}
\end{equation}
As before, the lower bound converges to zero for any $\alpha \in (0,1)$, hence,  in light of  \eqref{I-bar-short} and
\eqref{I-bar-short-2}, we find $\underset{\epsilon \rightarrow 0} \liminf \ \bar{I}_1 \geq 0$.\\
\\
Next, assume that $|y|^{\xi(d+\alpha)} > \epsilon e^{|\lambda_1|\bar{t}}$.
It follows from \eqref{takis4} that
$$ \bar{K}(\bar{x},y) \lesssim \frac{1}{\epsilon^{1/\xi}e^{|\lambda_1|\bar{t}/\xi}+|y|^{d+\alpha}},$$
while  \eqref{upper-on-u-over-steady} yields
\begin{align}
\bar{I}_2 &=
\int_{\lbrace{|y|>\rho , \ |\bar{x}+y| \leq R^{1/\epsilon} \rbrace}} (h(\bar{x},\bar{t})-h(\bar{x}+y,\bar{t}))
w(\bar{x}+y,\bar{t}) \bar{K}(\bar{x},y) dy
\nonumber \\
&\gtrsim -\int_{|x+y| \leq R^{1/\epsilon}} \frac{e^{\nu/\epsilon - |\lambda_1|\bar{t}} |\bar{x}+y|^{\xi(d+\alpha)}}
{1+e^{-|\lambda_1|\bar{t}} |\bar{x}+y|^{d+\alpha}} \frac{dy}{\epsilon^{1/\xi}e^{|\lambda_1|\bar{t}/\xi}+|y|^{d+\alpha}}
\nonumber \\
&\geq -e^{\nu/\epsilon} e^{(\xi-1)|\lambda_1|\bar{t}} e^{-\frac{\alpha |\lambda_1|\bar{t}}{d+\alpha}}
\int \frac{e^{-\xi |\lambda_1|\bar{t}} |\bar{x}+y|^{\xi(d+\alpha)}}{1+e^{-|\lambda_1|\bar{t}}
|\bar{x}+y|^{d+\alpha}} \frac{e^{-d|\lambda_1|\bar{t}/(d+\alpha)}dy}
{\epsilon^{1/\xi} e^{(1-\xi)|\lambda_1|\bar{t}/\xi}+e^{-|\lambda_1|\bar{t}}|y|^{d+\alpha}}
\nonumber \\
&= -\exp \left(\frac{\nu}{\epsilon} + (\xi-1)|\lambda_1| \bar{t} - \frac{\alpha |\lambda_1| \bar{t}}{d+\alpha} \right)
\int \frac{|\bar{x}e^{-|\lambda_1| \bar{t}/(d+\alpha)} + z|^{\xi(d+\alpha)}}
{1+|\bar{x}e^{-|\lambda_1| \bar{t}/(d+\alpha)} + z|^{d+\alpha}}
\frac{dz}{\epsilon^{1/\xi} e^{(1-\xi)|\lambda_1|\bar{t}/\xi} + |z|^{d+\alpha}} .
\label{I-bar-med}
\end{align}
Recall that $|\bar{x}| < e^{|\lambda_1|\bar{t}/(d+\alpha)}$.
The integral term above is finite so long as $\xi(d+\alpha) - (d+\alpha) < \alpha$, in which case
it is also bounded by $\epsilon^{-1/\xi} e^{(\xi-1)|\lambda_1|\bar{t} / \xi}$ uniformly in $\bar{x}$.
If $\nu$ is sufficiently small depending on $T$ but not on $\xi$, and $\xi$ is sufficiently small
depending on $\nu$, we have
$$ \frac{\nu}{\epsilon} + (\xi-1)|\lambda_1|\bar{t} - \frac{\alpha |\lambda_1| \bar{t}}{d+\alpha}
+ \left( 1-\frac{1}{\xi} \right) |\lambda_1| \bar{t} < 0 , $$
whence $\underset{\epsilon \rightarrow 0} \liminf \ \bar{I}_2 \geq 0$.\\
\\
For the remaining term, we use \eqref{rough-bounds} to conclude that
\begin{align}
\bar{I}_3 &=
\int_{|\bar{x}+y| > R^{1/\epsilon}} (h(\bar{x},\bar{t})-h(\bar{x}+y,\bar{t})) w(\bar{x}+y,\bar{t})
\bar{K}(\bar{x},y)dy
\nonumber \\
& \gtrsim -\int_{|\bar{x}+y| > R^{1/\epsilon}} \frac{e^{-|\lambda_1|\bar{t}}|\bar{x}+y|^{\xi(d+\alpha)}}
{1+ e^{-B_0 \bar{t}}|\bar{x}+y|^{d+\alpha}} \frac{dy}{1+|y|^{d+\alpha}}
\nonumber \\
& \geq -e^{(B_0-|\lambda_1|)\bar{t}} \int_{|y| > R^{1/\epsilon}/2} \frac{|\bar{x}+y|^{(\xi-1)(d+\alpha)}}
{1+|y|^{d+\alpha}} \gtrsim -e^{(B_0-|\lambda_1|)\bar{t}} R^{(-\alpha + (\xi-1)(d+\alpha))/\epsilon} .
\label{I-bar-long}
\end{align}
In light of the choice of $R$, the combined exponent above is negative and the lower bound converges to zero
as $\epsilon \rightarrow 0$. It then follows from \eqref{bar-I-decomp}, \eqref{I-bar-short}, \eqref{I-bar-short-2},
\eqref{I-bar-med}, and \eqref{I-bar-long} that
\begin{equation}
\liminf_{\epsilon \rightarrow 0} \bar{I} \geq 0 .
\label{liminf-bar-I}
\end{equation}
\\
Note that in order to estimate $\bar{I}_3$ we  needed $R$ to be substantially larger than $e^{|\lambda_1|\bar{t}/(d+\alpha)}$, depending on
$B_0$.\\
\\
To conclude the proof, fix  $\delta>0$ and choose $M$ large enough such that
$N(x, 1-\delta) > 3 |\lambda_1| / M$. If $W(x,t) < M(1-\delta)$, then $w(x,t) < 1-\delta$. Since $N(x,w)$ is
decreasing in $w$, it follows that
\begin{equation}
N(x,w) > 3 |\lambda_1|/M \ \text{ whenever } \ W < M(1-\delta) .
\label{intermediate-N-W}
\end{equation}
For all $T>0$, all $\epsilon$ sufficiently small, and all $t \in
[T/\epsilon, 2T/\epsilon]$, we know that $W(\cdot, t)$ achieves its minimum at some $\bar{x}\in \RR^d$.
It follows from \eqref{W-min}, \eqref{liminf-bar-I}, and \eqref{intermediate-N-W} that
$$ \max( W(\bar{x},t) - M(1-\delta), W_t(\bar{x},t) - W(\bar{x},t) |\lambda_1|/M) \geq 0 , $$
which implies that
$$ \lim_{t\rightarrow \infty} \inf_{|x|^{\xi(d+\alpha)} < e^{|\lambda_1| t}} W(x,t)
\geq M(1-\delta) . $$
For any $\xi>1$ and $M>0$, $h(\hat{x}|x|^{1/\epsilon},t/\epsilon)$ converges locally uniformly to $M$ if
$|x|^{\xi(d+\alpha)} < e^{|\lambda_1|t}$. Therefore,
$$ \lim_{t\rightarrow \infty} \inf_{|x|^{\xi(d+\alpha)} < e^{|\lambda_1| t}} w(x,t)
\geq 1-\delta . $$
Since $\delta>0$ was arbitrary and $\xi$ can be taken arbitrarily close to $1$, the second claim of
\eqref{limit-for-u} follows, and the proof is complete.\\
\\
We remark that the previous arguments  also establish the uniqueness of $u^+$ from the uniqueness of solutions to \eqref{main-eq}.
If $\bar{u}^+$ is another positive periodic steady state, the above argument would not be changed,
and we would conclude that $u^\epsilon(x,t) / u^+(\hat{x}|x|^{1/\epsilon})$ and $u^\epsilon(x,t) /
\bar{u}^+(\hat{x}|x|^{1/\epsilon})$ both converge locally uniformly to $1$ for $|x|^{d+\alpha} <
e^{|\lambda_1|t}$, which cannot happen if $u^+ \neq \bar{u}^+$.
\end{proof}
\appendix
\section{Viscosity Solutions and Generalized Limits}
\noindent{}Here we recall the classical definition, as it applies in the context of this paper,
for a viscosity solution to a variational inequality.
\begin{definition}
Let $\Psi_{(x_0,t_0)}$ be the collection of functions
$\psi: \RR^d \times [0,\infty) \rightarrow \RR$ that are smooth in a neighborhood of $(x_0,t_0)$ and
Lipschitz continuous everywhere else.\\
\\
A function $F: \RR^d \times [0,\infty) \rightarrow \RR$ is said to satisfy the variational inequality
$$ \min( F_t, F) \leq 0 $$
at the point $(x_0,t_0)$ in the \emph{viscosity sense} if, for all test functions $\psi \in \Psi_{(x_0,t_0)}$
such that $F-\psi$ achieves a global
maximum value of zero at $(x_0,t_0)$, that is, $\psi$ touches $F$ from above at $(x_0,t_0)$, then
$$ \text{either } \ F(x_0,t_0) \leq 0  \ \text{ or } \ \psi_t(x_0,t_0) \leq 0 . $$
Similarly, $F$ is said to satisfy the variational inequality
$$ \max( F_t, F) \geq 0 $$
at the point $(x_0,t_0)$ in the \emph{viscosity sense} if, for all test functions $\psi \in \Psi_{(x_0,t_0)}$
such that $F-\psi$ achieves a global
minimum value of zero at $(x_0,t_0)$, that is, $\psi$ touches $F$ from below at $(x_0,t_0)$, then
$$ \text{either } \ F(x_0,t_0) \geq 0  \ \text{ or } \ \psi_t(x_0,t_0) \geq 0 . $$
\label{def-viscosity}
\end{definition}
\noindent{}We also recall the definition of generalized upper and lower locally uniform limits.
\begin{definition}
For a bounded family of functions $\lbrace{ v^\epsilon \rbrace}_{\epsilon>0}$ with each
$v^\epsilon: \RR^d \times [0,\infty) \rightarrow \RR$ smooth, the half-relaxed upper limit is given by
$$ v^*(x,t) := \limsup_{\epsilon \rightarrow 0 , \ (y,s) \rightarrow (x,t)}
v^{\epsilon}(y,s) . $$
Similarly, the half-relaxed lower limit is given by
$$ v_*(x,t) := \liminf_{\epsilon \rightarrow 0, \ (y,s) \rightarrow (x,t)}
v^\epsilon (y,s) . $$
\label{def-lims}
\end{definition}
\noindent{}The half-relaxed upper and lower limits are, respectively,
upper- and lower-semicontinuous. As an immediate consequence of the definition, it follows that,
on any compact set $K \subset \RR^d \times \RR_+$,
\begin{equation}
\sup_{K} (v^\epsilon-v^*) \rightarrow 0 \ \text{ and } \ \inf_{K} (v^\epsilon-v_*) \rightarrow 0
\ \text{ uniformly in } \epsilon .
\label{loc-unif-limits}
\end{equation}
\noindent{}Lastly, we recall the definitions for viscosity sub- and super-solutions
used in Sections 4 and 5; see Barles and Imbert \cite{BI}.
\begin{definition}
An upper semicontinuous function $w_0$ is a sub-solution to \eqref{eq-v-eps} if,
given a smooth test function
$\phi$ such that $w_0 - \phi$ assumes a maximum value of $0$ at $(x_0,t_0)$, there exists
a family of open balls $B_{\delta_n}(x_0)$ with $\delta_n \rightarrow 0$ such that, for
all $n$,
\begin{align}
\phi_t & (x_0,t_0)+ \int_{B_{\delta_n}(x_0)} \left(1- \frac{\exp(\phi(\eta^\epsilon
(x_0,y),t_0)/\epsilon)}{\exp(\phi(x_0,t_0)/\epsilon)} \right) K dy \nonumber \\
&+ \int_{B_{\delta_n}(x_0)^c} \left(1- \frac{\exp(w_0(\eta^\epsilon
(x_0,y),t_0)/\epsilon)}{\exp(w_0(x_0,t_0)/\epsilon)} \right) K dy \leq
\frac{f(\hat{x_0}|x_0|^{1/\epsilon},\exp(w_0(x_0,t_0)/
\epsilon))}{\exp(w_0(x_0,t_0)/\epsilon)} .
\label{subsolution}
\end{align}
A lower semicontinuous function $w_1$ is a super-solution to \eqref{eq-v-eps} if,
given a smooth
test function $\phi$ such that $w_1 - \phi$ assumes a minimum value of $0$ at $(x_1,t_1)$,
there exists a family of open balls as above such that
\begin{align}
\phi_t & (x_1,t_1)+ \int_{B_{\delta_n}(x_1)} \left(1- \frac{\exp(\phi(\eta^\epsilon
(x_1,y),t_1)/\epsilon)}{\exp(\phi(x_1,t_1)/\epsilon)} \right) K dy \nonumber \\
&+ \int_{B_{\delta_n}(x_1)^c} \left(1- \frac{\exp(w_1(\eta^\epsilon
(x_1,y),t_1)/\epsilon)}{\exp(w_1(x_1,t_1)/\epsilon)} \right) K dy \geq
\frac{f(\hat{x_1}|x_1|^{1/\epsilon},\exp(w_1(x_1,t_1)/
\epsilon))}{\exp(w_1(x_1,t_1)/\epsilon)} .
\label{supersolution}
\end{align}
\label{def-visc-sub-sup}
\end{definition}

\section{Principal Eigenvalue for the Linearized Operator and Positive Steady States}
\noindent{}We show that the operator $L^\alpha - \mu$ has a positive periodic first eigenfunction.
A similar claim is made in \cite{BRR} for the case where $L^\alpha = (-\Delta)^{\alpha/2}$,
but the authors mainly show a Rayleigh-type formula for the principal eigenfunction. For
completeness, we provide here a proof outline for the following claim:
\begin{proposition}
For $\mu$ a smooth periodic function and $L^\alpha$ as in \eqref{L-alpha} with a kernel $K$
satisfying \eqref{takis3} and \eqref{takis4}, the operator $L^\alpha-\mu$ has a unique
positive periodic eigenfunction 
with eigenvalue $\lambda_1$, that is
\begin{equation}
L^\alpha e^g - \mu e^g = \lambda_1 e^g .
\label{eq-eigen}
\end{equation}
The eigenvalue $\lambda_1$ is simple in the algebraic and geometric sense, and is the bottom
of the spectrum for $L^\alpha - \mu$.
\label{prop-eigen}
\end{proposition}
\noindent{}To handle the issue of periodicity,
we construct this eigenfunction on the torus $\TT^d = [0,1]^d$.
If $u$ is a smooth $1$-periodic function on $\RR^d$, then
it is equivalent to a smooth function defined on $\TT^d$. Moreover,
\begin{equation}
L^\alpha[u](x) = \int_{\RR^d} K(x,y) (u(x) - u(x+y)) dy =
\int_{\TT^d} \tilde{K}(x,y) (u(x) - u(x+y)) dy,
\label{periodic-L}
\end{equation}
where
$$ \tilde{K}(x,y) := \sum_{k \in \ZZ^d} K(x,y+k);$$
note that, in view of  \eqref{takis4}, the sum above converges for all $x \in \TT^d$ and $y \in \TT^d \setminus \lbrace{0\rbrace}$.\\
\\
It also follows from \eqref{takis3} and \eqref{takis4} that, for a constant $C>0$, $\mathcal{\tilde{K}} = \tilde{K}$,
$|D_x \tilde{K}|$, or $|D^2_x \tilde{K}|$, and  all $x$ and $y$,
\begin{equation}\label{K-tilde-bounds}
\tilde{K} \ \text{ is bounded from below, symmetric in }  \ y \
\text{and } \ C^{-1} \leq \mathcal{\tilde{K}}(x,y)|y|^{d+\alpha} \leq C.
\end{equation}
Proposition \ref{prop-eigen} then follows from the following lemma:
\begin{lemma}
Let $\mu$ a smooth function on $\TT^d$ and $L^\alpha$ as in \eqref{periodic-L}
satisfying \eqref{K-tilde-bounds}.
Then the operator $L^\alpha-\mu$ has a unique positive eigenfunction $e^g$ satisfying
\begin{equation}
 L^\alpha e^g - \mu e^g = \lambda_1 e^g .
\label{eq-eigen-2}
\end{equation}
The eigenvalue $\lambda_1$ is simple in the algebraic
and geometric sense, and is the bottom of the spectrum for
$L^\alpha - \mu$.
\label{prop-eigen-2}
\end{lemma}
\begin{proof}
We appeal to the Krein-Rutman Theorem \cite{K-R}.
For $X$ a Banach
space and $K: X \rightarrow X$ a positive operator, if the spectral radius
$r(K)$ is strictly positive, then $r(K)$ is an eigenvalue of $K$ with a positive
eigenfunction. We construct $K$ as an inverse to $L^\alpha$ after
a suitable shift to gain coercivity.\\
\\
We first show that $L^\alpha$ satisfies a G\r{a}rding-type inequality.
For $u \in C^\infty(\TT^d)$, we write 
\begin{equation}\label{L_2}
\langle L^\alpha[u],u \rangle = L_1 + L_2,
\end{equation}
where 
$$
L_1:=\int \int \tilde{K}(x,y)(u(x)-u(x+y))^2 dy dx \ \text{ and } \ L_2:=\int \int \tilde{K}(x,y)(u(x)-u(x+y))u(x+y) dy dx.
$$
It follows from \eqref{K-tilde-bounds} that
\begin{equation}
\| u \|_{\dot{H}^{\alpha/2}(\TT^d)}^2 \gtrsim
L_1 \approx \int \int \frac{(u(x)-u(x+y))^2}{|y|^{d+\alpha}} dx dy \gtrsim
\| u \|_{\dot{H}^{\alpha/2}(\TT^d)}^2;
\label{L_1}
\end{equation}
here  $\| \cdot \|_{\dot{H}^s(\TT^d)}$ and $\| \cdot \|_{H^s(\TT^d)}$ denote, respectively,
the homogeneous and inhomogeneous Sobolev norm, that is,
$$ \| u \|_{\dot{H}^s(\TT^d)} := \left( \int_{\TT^d} |\Lambda^s u (y)|^2 dy \right)^{1/2} \ \text{ and }
\ \| u \|_{H^s(\TT^d)} := \| u \|_{\dot{H}^s(\TT^d)} + \| u \|_{L^2(\TT^d)} . $$
We need to show that the remainder $L_2$ is bounded by $u$ in $L^2(\TT^d)$. Changing variables
and using \eqref{K-tilde-bounds} we find 
\begin{align*}
L_2 &= \int \int \tilde{K}(x+y,y)(u(x+y)-u(x))u(x) dy dx \\
&= \frac{1}{2} \int \int (\tilde{K}(x,y) - \tilde{K}(x+y,y)) (u(x)-u(x+y))u(x) dy dx \\
&= \frac{1}{2} \int \int (2 \tilde{K}(x,y) - \tilde{K}(x+y,y) - \tilde{K}(x-y,y)) u(x)^2 dy dx \\
&\gtrsim -\int \int |y|^2 \| D^2_x K \|_{L^\infty} |y|^{-d-\alpha} u(x)^2 dy dx .
\end{align*}
Since $|y|^{2-d-\alpha}$ is integrable on $\TT^d$ for all $\alpha \in (0,2)$,
we get that, for some fixed positive constants $\tilde{c}$, $\tilde{C}$, and $\tilde{D}$,
\begin{equation}
\tilde{D} \| u \|_{H^{\alpha/2}} \geq
\langle L^\alpha[u], u \rangle \geq \tilde{c} \| u \|_{\dot{H}^{\alpha/2}}^2 -
\tilde{C} \| u \|_{L^2}^2 .
\label{Garding}
\end{equation}
Then, for a constant $\mu_0$ sufficiently large, the operator $ \mathcal{L} := L^\alpha - \mu + \mu_0$
is positive and satisfies the hypotheses of the Lax-Milgram theorem in $H^{\alpha/2}(\TT^d)$.
Hence it is bijective and, by Sobolev embedding, has a compact everywhere defined inverse
$$ K:= \mathcal{L}^{-1} =
(L^\alpha - \mu + \mu_0)^{-1} : H^{\alpha/2}(\TT^d) \rightarrow Dom(L^\alpha) . $$
By the Krein-Rutman Theorem, there exists a positive eigenfunction $e^g$ of $K$
with eigenvalue $r(K)$. Therefore,
$$ L^\alpha[e^g] - \mu e^g = \left( \frac{1}{r(K)} - \mu_0 \right) e^g . $$
This is precisely \eqref{eq-eigen-2} with $\lambda_1 = r(K)^{-1}-\mu_0$. Evidently, $\lambda_1$ is the bottom eigenvalue for
the spectrum of $L^\alpha-\mu$.\\
\end{proof}
\noindent{}When the principal eigenvalue is negative, we also find nontrivial steady states
for \eqref{main-eq}. We sketch a proof of this fact in the context of the
periodic and nonlocal model \eqref{main-eq}.
\begin{proposition}
Let $\tilde{K}$ be given by \eqref{periodic-L} and $f$ be as in \eqref{takis2} and
assume that $\lambda_1 < 0$.
Then there exists a positive smooth function $u^+ : \TT^d \rightarrow \RR$ satisfying
\begin{equation}
L^\alpha[u^+]= f(x,u^+) .
\label{steady-equation}
\end{equation}
\label{prop-steady-state}
\end{proposition}
\begin{proof}
Let $M$ be the constant given in \eqref{takis2}. Then $\overline{u} := M$ is a supersolution to \eqref{steady-equation}.
Since $\lambda_1 < 0$, there exists a sufficiently small $\delta>0$ so that
$$\mu \delta e^g + \lambda_1 \delta e^{g} \leq f(x, \delta e^{g}) . $$
Then $\underline{u} := \delta e^g$ is a positive subsolution.\\
\\
Let $N_0 = \max_{x\in \TT^d, \ u\in \RR} (-\partial_u f(x,u))$ and note that $N_0>0$.
We then take $u_0 = \underline{u}$ and obtain the sequence $\lbrace{ u_k \rbrace}_{k\in \ZZ}$ by iteratively solving
\begin{equation}
L^\alpha[u_{k+1}]+N_0u_{k+1} = f(x,u_k)+N_0u_k .
\label{steady-iteration}
\end{equation}
It follows from  the maximum principle that $u_1 \geq u_0$. Arguing by induction and using the maximum principle we get that  $u_k \leq \overline{u}$ and $u_{k+1} \geq u_k$ for all $k$.\\
\\
There must then exist a pointwise limit $u^+$ on $\TT^d$ which is bounded from 
below by $u_0$, and, hence, is positive, smooth and satisfies $ L^\alpha[u^+]= f(x,u^+)$.
\end{proof}

\section{Proof of Lemma \ref{lem-ACC}}
\noindent{}We prove here a robust bound on the size of the diffusion $L^\alpha$ applied to
a particular function, used in the construction of the relaxed upper and lower bounds of
Proposition \ref{prop-rough}. Recall that $h(x,t)= (1+e^{-\lambda t}|x|^{d+\alpha})^{-1}$.\\
\\
The main observation is that the derivatives of $h$ decay algebraically in $|x|$.
That is,
$$ \partial_i h(x,t) = -(d+\alpha) e^{-\lambda t} x_i |x|^{d+\alpha-2} h^2 $$
and
$$ \partial_{ij}^2h(x,t) = -(d+\alpha) e^{-\lambda t} h^2 |x|^{d+\alpha-2}
\left( \delta_{ij} + (d+\alpha-2) \frac{x_i x_j}{|x|^2} - 2 (d+\alpha) e^{-\lambda t} x_i x_j
|x|^{d+\alpha-2} h \right) . $$
It follows from Taylor's Theorem that, for some  constants independent of $\lambda$, $x$ and  $t$,
\begin{equation}
| h(x,t) - h(x+y,t) + Dh (x,t) \cdot y| \lesssim \sup_{|z|<|y|}
\frac{e^{-2 \lambda t/(d+\alpha)}}
{(1+e^{-\lambda t}|x+z|^{d+\alpha})^{\frac{d+\alpha+2}{d+\alpha}}} |y|^2 .
\label{last-short-bound}
\end{equation}
Let
$$R = \frac{1}{2} \left( \frac{e^{\lambda t}}{2} \right)^{1/(d+\alpha)}
\left( 1+e^{-\lambda t}|x|^{d+\alpha} \right)^{1/(d+\alpha)} , $$
and observe that
$$ L^\alpha[h] = L_1[h] + L_2[h],$$ 
where
$$L_1[h] := \int_{|y| \leq R} K(x,y)(h(x,t)-h(x+y,t)) dy \ \text{ and } \ 
L_2[h]:=\int_{|y| > R} K(x,y)(h(x,t)-h(x+y,t)) dy . $$
Recall \eqref{takis4} and observe that  it follows from \eqref{last-short-bound} that
\begin{align*}
&|L_1[h]| = \left| \int_{|y| \leq R} K(x,y)(h(x,t)-h(x+y,t)) dy \right| \lesssim
\int_{|y| \leq R} \| \nabla^2 h(\cdot, t) \|_{B_R(x)} |y|^{2-d-\alpha} dy \\
& \hspace{2cm} \lesssim R^{2-\alpha} e^{-2 \lambda t/(d+\alpha)}
\left( 1+e^{-\lambda t}\max(0,|x|-R)^{d+\alpha} \right)^{-\frac{d+\alpha+2}{d+\alpha}} .
\end{align*}
As long as $e^{-\lambda t}|x|^{d+\alpha} > 1$, we have that $R < |x|/2$. Therefore,
\begin{equation}
|L_1[h]| \lesssim
\begin{cases}
e^{-\alpha \lambda t/(d+\alpha)} (1+e^{-\lambda t}|x|^{d+\alpha})^{-1 -\alpha/(d+\alpha)},
& \text{if} \ e^{-\lambda t}|x|^{d+\alpha} > 1 \\
e^{-\alpha \lambda t/(d+\alpha)}, & \text{otherwise}.
\end{cases}
\label{L1-ACC}
\end{equation}
For  $L_2[h]$, we have
$$ \left| \int_{|y|>R} K(x,y) h(x,t) dy \right| \lesssim h(x,t) R^{-\alpha}
\lesssim \frac{e^{-\alpha \lambda t/(d+\alpha)}}{(1+e^{-\lambda t}|x|^{d+\alpha})^{1+\alpha/(d+\alpha)}} , $$
and
$$ \left| \int_{|y|>R} K(x,y) h(x+y,t) dy \right| \lesssim R^{-d-\alpha}
\int \frac{dy}{1+e^{-\lambda t} |x+y|^{d+\alpha}} \lesssim R^{-d-\alpha} e^{-\lambda t d / (d+\alpha)}
\lesssim \frac{e^{-\alpha \lambda t/(d+\alpha)}}{1+e^{-\lambda t}|x|^{d+\alpha}} . $$
The bounds above, combined with \eqref{L1-ACC} establish \eqref{ACC-bound}, completing the proof.

\end{document}